\documentclass{amsart}
\usepackage{amsmath}
\usepackage{amssymb} 
\usepackage{amscd} 
\usepackage{graphicx}   %\usepackage[dvipdfmx]{graphicx}
\usepackage{xcolor}
\usepackage{marvosym}
\usepackage{float}
\usepackage{ascmac}
\usepackage{wrapfig}
\usepackage{epsfig}
\usepackage{tikz}
\usepackage[all]{xy}
\usepackage{hyperref}

\hypersetup{
	colorlinks=true,
	linktoc=all,
    	linkcolor={red!50!black},
   	citecolor={blue!50!black},
  	urlcolor={blue!80!black}
	} %%% This tones down the bright green hyperlink box and its other colors...

\allowdisplaybreaks % To make pagebreaks in equations

\usepackage{booktabs} %for nice tables

\newtheorem{theorem}{Theorem}[section]
\newtheorem{lemma}[theorem]{Lemma}

\newtheorem{proposition}[theorem]{Proposition}
\newtheorem{corollary}[theorem]{Corollary}
\newtheorem{claim}[theorem]{Claim}

\newtheorem{remark}[theorem]{Remark}

\theoremstyle{definition}
\newtheorem{definition}[theorem]{Definition}
\newtheorem{example}[theorem]{Example}

\numberwithin{equation}{section}
\numberwithin{figure}{section}
\numberwithin{table}{section}

\renewcommand{\labelenumi}{ $(\arabic{enumi})$ }

\begin{document}
\baselineskip 14pt

\title[Every non-trivial knot group is fully residually perfect]{Every non-trivial knot group is\\ fully residually perfect}

\author[T. Ito]{Tetsuya Ito}
\address{Department of Mathematics, Kyoto University, Kyoto 606-8502, Japan}
\email{tetitoh@math.kyoto-u.ac.jp}
\thanks{TI has been partially supported by JSPS KAKENHI Grant Number JP21H04428 and JP23K03110.}

\author[K. Motegi]{Kimihiko Motegi}
\address{Department of Mathematics, Nihon University, 
3-25-40 Sakurajosui, Setagaya-ku, 
Tokyo 156--8550, Japan}
\email{motegi.kimihiko@nihon-u.ac.jp}
\thanks{KM has been partially supported by JSPS KAKENHI Grant Number JP25K07018, JP21H04428, JP23K03110, JP23K20791, JP26K00606 and Joint Research Grant of Institute of Natural Sciences at Nihon University for 2026}

\author[M. Teragaito]{Masakazu Teragaito}
\address{Department of Mathematics and Mathematics Education, Hiroshima University, 
1-1-1 Kagamiyama, Higashi-Hiroshima, 739--8524, Japan}
\email{teragai@hiroshima-u.ac.jp}
\thanks{MT has been partially supported by JSPS KAKENHI Grant Number JP20K03587 and JP25K07004.}

\begin{abstract}
Given a class $\mathcal{P}$ of groups we say that a group $G$ is {\em fully residually}  $\mathcal{P}$ if 
for any finite subset $F$ of $G$, 
there exists an epimorphism from $G$ to a group in $\mathcal{P}$ 
which is injective on $F$. 
It is known that any non-trivial knot group is fully residually finite, but not residually free.
For hyperbolic knots, its knot group is fully residually closed hyperbolic $3$--manifold group, and such a group is fully residually simple. 
In this article, 
we show that every non-trivial knot group is fully residually perfect, closed $3$--manifold group.  
\end{abstract}

\maketitle

{
\renewcommand{\thefootnote}{}
\footnotetext{
}

\setcounter{tocdepth}{1}
\tableofcontents

\section{Introduction}
\label{introduction}
Given a class $\mathcal{P}$ of groups we say that a group $G$ is {\em residually}  $\mathcal{P}$ if 
for a given non-trivial element $g \in G$, 
there exists an epimorphism $\varphi$ from $G$ to a group in $\mathcal{P}$ such that 
$\varphi(g) \ne 1$. 
For instance, 
when $\mathcal{P}$ is the class of finite groups, $G$ is said to be {\em residually finite}. 
More generally, $G$ is said to be {\em fully residually} $\mathcal{P}$ if for every finite subset $F$ of $G$, 
there exists an epimorphism $\varphi$ from $G$ to a group in $\mathcal{P}$ such that 
$\varphi$ is injective on $F$. 

For a non-trivial knot $K$ in the $3$--sphere $S^3$, 
the fundamental group of its exterior is called the \textit{knot group} of $K$, which we denote by $G(K)$. 
It is known that every knot group is residually finite \cite{Hem_residual_finite}, and 
it immediately implies that every knot group is fully residually finite; see \cite[Proposition~9.2]{IMT_realization} for instance. 
It is also known that every non-trivial knot group cannot be residually free \cite{Wilton}. 

This article is motivated by the following theorem, which is explicitly stated in Aschenbrenner--Friedl--Wilton's book 
\cite[p.66 E.5]{AFW}. 

\begin{theorem}[\cite{AFW}]
\label{hyp_fully}
Let $K$ be a hyperbolic knot.   
Then its knot group is fully residually closed hyperbolic $3$--manifold group. 
\end{theorem}

As mentioned there, this theorem is a direct consequence of a more precise result
below obtained by Groves-Manning \cite[Theorem~9.7]{GM} and Osin \cite[Theorem~1]{Osin} independently; 
see also \cite{IchiharaMT}. 

\begin{theorem}[\cite{GM,Osin,IchiharaMT}]
\label{GMO}
Let $K$ be a hyperbolic knot.   
Then every non-trivial element remains non-trivial for all but finitely many Dehn fillings. 
\end{theorem}

On the other hand, 
the knot group of a cabled knot, which may be a torus knot (by allowing the trivial companion), or a non-prime knot turns out to have a non-trivial element which becomes trivial for infinitely many Dehn fillings. 
The next theorem shows that such a phenomenon can only occur for these knots. 

\begin{theorem}
\label{prime_non-cable}
Let $K$ be a non-trivial knot. 
Then the following two conditions are equivalent.
\begin{enumerate}
\item
Every non-trivial element of $G(K)$ remains non-trivial for all but finitely many Dehn fillings.
\item
$K$ is a prime, non-cabled knot. 
\end{enumerate}
\end{theorem}

Furthermore, even when $K$ is a non-prime knot or a cabled knot (possibly a torus knot), we have the following result. 

\begin{theorem}
\label{1/n}
Let $K$ be a non-trivial knot. 
Given any finite subset $F \subset G(K) - \{ 1 \}$, 
there exists a constant $N_F$ such that 
$g$ remains non-trivial after $1/n$--Dehn filling for all $g \in F$ provided $n \ge N_F$. 
\end{theorem}

In particular, as we expected, the following holds. 

\begin{corollary}
\label{vanish_for _all}
Let $K$ be a non-trivial knot and $g$ an element of $G(K)$. 
Then $g$ vanishes for all the non-trivial Dehn fillings if and only if $g$ is the identity element.  
\end{corollary}

The next result is a straightforward consequence of Corollary~\ref{vanish_for _all}. 

\begin{corollary}
\label{Dehn_filling_homo}
The ``universal" Dehn filling homomorphism 
\[
\Phi \colon G(K) \to \prod_{r \in \mathbb{Q}} \pi_1(K(r))
\]
is injective. 
\end{corollary}

\begin{remark}
More precisely, 
following Theorem~\ref{1/n}
we may embed $G(K)$ in $\displaystyle\prod_{n\in \mathbb{Z}} \pi_1(K(1/n))$. 
\end{remark}

Let $K$ be a non-trivial knot with the exterior $E(K)$. 
Every essential simple closed curve on $\partial E(K)$ represents a {\em slope element} $\gamma$. 
Denote by $\langle\!\langle \gamma \rangle\!\rangle$ the normal closure of $\gamma$. 
Note that both $\gamma$ and its inverse $\gamma^{-1}$ define the same normal closure 
$\langle\!\langle \gamma \rangle\!\rangle = \langle\!\langle \gamma^{-1} \rangle\!\rangle$, and correspond to a slope $r \in \mathbb{Q}$, 
so it is reasonable to write $\langle\!\langle r \rangle\!\rangle = \langle\!\langle \gamma \rangle\!\rangle$. 
It is easy to see that  $\langle\!\langle r \rangle\!\rangle$ consists of elements of $G(K)$ which becomes trivial after $r$--Dehn fillings. 

In the language of knot group (without using Dehn filling), 
Theorem~\ref{prime_non-cable} and Corollary~\ref{vanish_for _all} have the following reformulation. 

\begin{corollary}\
\label{normal_closure}
For any non-trivial knot, we have
\[
\bigcap_{r \in \mathbb{Q}} \langle\!\langle r \rangle\!\rangle = \{ 1 \}.
\]
Furthermore, 
\[
\bigcap_{r \in S} \langle\!\langle r \rangle\!\rangle = \{ 1 \}\quad \textrm{for any infinite subset $S \subset \mathbb{Q}$}
\]
 if and only if $K$ is a prime, non-cabled knot. 
\end{corollary}

This result should be compared with the following remark.  

\begin{remark}
\label{large_intersection}
For arbitrary finite subset $S \subset \mathbb{Q}$, 
\[
\bigcap_{r \in S} \langle\!\langle r \rangle\!\rangle \ne \{ 1 \}
\]
except when $K$ is a torus knot $T_{p, q}$ and $pq \in S$; see \cite[Theorem~1.4]{IMT_Magnus}. 
\end{remark}

As we will see in Subsection~\ref{hyperbolicity}, not all knot groups are residually closed hyperbolic 3-manifold groups. 
However, applying Theorem~\ref{1/n}, we can extend Theorem~\ref{hyp_fully} to all non-trivial knot groups, after relaxing the condition for the target $3$--manifold groups.

\begin{corollary}
\label{closed_3_manifold_group}
Every non-trivial knot group is fully residually closed $3$--manifold group.  
\end{corollary}

Long-Reid proved that hyperbolic knot groups (more generally, hyperbolic 3-manifold groups) are fully residually simple (\cite{LR}, see also \cite[(C8)]{AFW}).
On the other hand, torus knot groups cannot be residually simple since they have non-trivial centers,  
and \cite[Question 7.3.4]{AFW} (for knot group case) asks the residually simpleness for satellite knot groups. 
Note that a non-abelian simple group is perfect. 
In this prospect, it is natural to ask whether all non-trivial knot groups are (fully) residually perfect. 
Since $1/n$--surgery yields an integral homology $3$--sphere, which has perfect fundamental group, 
Theorem~\ref{1/n} gives an affirmative answer. 

\begin{corollary}
\label{residually_perfect}
Every non-trivial knot group is fully residually perfect, closed $3$--manifold group. 
\end{corollary}

\section{Fully residual property}
\label{fully residual property}

In this section, 
we present some general results on fully residual property. 
The first one may be well known, but we provide its proof for completeness. 
In the following, let $\mathcal{P}$ be any class of groups. 

\begin{proposition}
\label{fully_residual_P}
A group $G$ is fully residually $\mathcal{P}$ if and only if 
for given finitely many non-trivial elements $g_1, \dots, g_m$ of $G$, 
there exists an epimorphism $\psi$ from $G$ to some group in $\mathcal{P}$ such that 
$\psi(g_i) \ne 1$ for $1 \le i \le n$. 
\end{proposition}

\begin{proof}
Assume that $G$ is fully residually $\mathcal{P}$. 
Then for a given finite set $F = \{ g_1, \dots, g_m \}$, 
let us put $\widehat{F} = \{1, g_1, \dots, g_m \}$. 
Since $G$ is fully residually $\mathcal{P}$, 
for the finite subset $\widehat{F}$, 
we have an epimorphism $\varphi$ from $G$ to some group in $\mathcal{P}$ 
which is injective on $\widehat{F}$, in particular, $\varphi(g_i) \ne 1$ for every element of $F$. 
Conversely, 
let us take a finite subset $\{ h_1, \dots, h_n \}$. 
For a finite subset $\{ h_ih_j^{-1}\}_{1 \le i < j \le n}$, 
by the assumption we have an epimorphism $\psi$ from $G$ to some group in $\mathcal{P}$ such that 
$\psi(h_ih_j^{-1}) \ne 1$, i.e. $\psi(h_i) \ne \psi(h_j)$ for all $1 \le i < j \le n$. 
This means that $\psi$ is injective on $\{ h_1, \dots, h_n \}$.  
\end{proof}

\medskip

Following this we will mainly use the latter equivalent form as the definition of $G$ being fully residually $\mathcal{P}$. 

\medskip

In general, a residually $\mathcal{P}$ group may not be fully residually $\mathcal{P}$. 
However, for knot groups $G(K)$, 
if $K$ is not a torus knot, 
these two properties are equivalent. 
To prove this , we begin by recalling \cite[Lemma~5.1]{IMT_Magnus}, in which $y$ is assumed to be a slope element, but the proof works for a non-central element. 

\begin{lemma}[\cite{IMT_Magnus}]
\label{Magnus_lemma}
Let $x, y$ be elements in $G(K)$.
If $x, y \not\in Z(G(K))$, then we can take an element $h \in G(K)$ so that $[h x h^{-1}, y] \ne 1$.  
\end{lemma}

\begin{proposition}
\label{fully_residual=residual}
Let $K$ be a knot which is not a torus knot. 
Then $G(K)$ is residually $\mathcal{P}$ if and only if it is fully residually $\mathcal{P}$. 
\end{proposition}

\begin{proof}
Assume that $G(K)$ is residually $\mathcal{P}$. 
Suppose for a contradiction that it is not fully residually $\mathcal{P}$. 
Then there exists a finite subset $F =  \{ g_1, \dots, g_m \}$ of non-trivial elements of $G$ satisfying: 

\smallskip

$(\ast)$ For any group $H$ in $\mathcal{P}$ and any epimorphism $q \colon G(K) \to H$, 
$q(g_i) = 1$ for some $g_i \in F$. 

\smallskip

Let us pick such a finite subset $F$ so that $|F| = m$ is minimal. 
Since $G(K)$ is assumed to be residually $\mathcal{P}$, $m \ge 2$. 
Then, by definition, we have: 

\smallskip

$(\ast\ast)$ For any $m-1$ non-trivial elements $x_1, \dots, x_{m-1}$ of $G(K)$, 
we have a group $H'$ in $\mathcal{P}$ and an epimorphism $q' \colon G(K) \to H'$ such that 
$q'(x_i) \ne 1$ for $1 \le i \le m-1$. 

\smallskip

Since $K$ is not a torus knot, $Z(G(K)) = \{ 1 \}$. 
Apply this Lemma~\ref{Magnus_lemma} to $x = g_i,\ y=g_{i+1}\ (1 \le i \le m-1)$ to see that 
there exists an element $h_i$ such that 
$[h_i g_i h_i^{-1}, g_{i+1}] \ne 1$ for $i = 1, \dots, m-1$. 
Let us write $x_1 = [h_1 g_1 h_1^{-1}, g_2],\dots, x_{m-1} = [h_{m-1} g_{m-1} h_{m-1}^{-1}, g_m]$. 
Following $(\ast\ast)$ we have a group $H'$ and an epimorphism $q' \colon G(K) \to H'$ such that 
$q'(x_i) \ne 1$ for $i =1, \dots, m-1$. 
On the other hand, 
$(\ast)$ shows that $q'(g_j) = 1$ for some $1 \le j \le m$, 
which implies $q'(x_{j-1}) = 1$; if $j =1 $, then $q'(x_1) = 1$. 
This is a contradiction. 
Hence, $G(K)$ is fully residually $\mathcal{P}$. 
\end{proof}

\medskip 

Our proof of Theorem~\ref{1/n} does not require Proposition~\ref{fully_residual=residual}, 
but this fact has its own interest. 

\section{Decomposing the knot group and its elements}
\label{fundamental_groups}

Let $K$ be a non-trivial knot in the $3$--sphere $S^3$ with exterior $E(K)$. 
Consider the torus decomposition of the exterior $E(K)$, 
and let $X$ be the outermost decomposing piece containing $T = \partial E(K)$. 
Denote $\partial X = T \cup T_1 \cup \cdots \cup T_n$. 
Note that each $T_i$ bounds a non-trivial knot exterior $E_i$ in $E(K)$ for $1 \le i \le n$. 
Then $E(K)$ is expressed as  $E(K) = X \cup E_1 \cup \cdots \cup E_n$. 
If $K$ is a torus knot or a hyperbolic knot, 
then $X = E(K)$ with $\partial X = T$. 
Throughout this section we assume that $E(K)$ admits a non-trivial torus decomposition, 
i.e. $K$ is a satellite knot. 

\subsection{Convention on the fundamental groups}
\label{convention}

Throughout this paper, 
we choose a point $p$ on $T = \partial E(K)$ as the base-point of $\pi_1(E(K))$.  

For each $E_i$ and $T_i$ we take a point $p_i \in T_i \subset E_i$.  
Since each $T_i$ is incompressible in $E(K)$, 
the Loop Theorem says that 
\[
\varphi \colon \pi_1(T, p) \to \pi_1(X, p),\; \varphi_i \colon \pi_1(T_i, p_i) \to \pi_1(X, p_i)\
\textrm{and}\  \psi_i \colon \pi_1(T_i, p_i) \to \pi_1(E_i, p_i)
\]
are injective, 
and thus 
\[
\Phi \colon \pi_1(X, p) \to \pi_1(E(K), p)\quad  \textrm{and}\quad  \Phi _i \colon \pi_1(E_i, p_i) \to \pi_1(E(K), p_i)
\]
are also injective, 
where $\varphi, \varphi_i, \psi_i, \Phi, \Phi_i$ denote monomorphisms induced by the inclusions. 
For notational simplicity, we will denote $\pi_1(E(K), p)$ by $\pi_1(E(K)) = G(K)$.  
We also use $\pi_1(X)$ to denote $\pi_1(X, p)$.  

Let us describe how to regard 
$\pi_1(T_i, p_i)$ as a subgroup of $\pi_1(X)$. 
Choose and fix a path $u_i$ in $X$ from the base point $p$ to $p_i$, which gives an isomorphism 
$u_{i, \ast} \colon \pi_1(X, p_i) \to \pi_1(X, p)$ 
sending $[c_i] \in \pi_1(T_i, p_i)$ to $u_{i, \ast}([c_i]) = [u_i \ast c_i \ast \overline{u_i}] \in \pi_1(T_i, p)$, 
where $\ast$ denotes the concatenation and $\overline{u_i}$ denotes $u_i$ with opposite orientation.
Now we put $\pi_1(T_i) = u_{i, \ast}(\varphi_i(\pi_1(T_i, p_i))) \subset u_{i, \ast}(\pi_1(X, p_i)) = \pi_1(X)$.  
Also we identify $\pi_1(X)$ with the subgroup $\Phi(\pi_1(X)) \subset \pi_1(E(K))$. 
This identification also specifies the subgroup $\pi_1(T_i) = \Phi(u_{i, \ast}(\varphi_i(\pi_1(T_i, p_i))))
\subset \pi_1(E(K))$. 
For $\pi_1(E_i, p_i)$, 
using the pre-specified path $u_i$ in $X$, 
we set $\pi_1(E_i) = u_{i, \ast}(\Phi_i(\pi_1(E_i, p_i))) \subset \pi_1(E(K))$. 
Note that for $T_i \subset E_i$, 
$u_{i, \ast}(\Phi_i(\psi_i(\pi_1(T_i, p_i)))) \subset \pi_1(E(K), p)$ coincides with $\pi_1(T_i) \subset \pi_1(E(K))$. 
The above convention gives precise specification of subgroups 
$\pi_1(X), \pi_1(E_i)$ and $\pi_1(T_i)$ in $\pi_1(E(K)) = G(K)$. 
We call $\pi_1(X)$ and $\pi_1(E_i)$ the factor groups of $G(K)$. 

It should be noted here that for another path $w_i$ in $X$ from $p$ to $p_i$, 
we have 
\[
w_i \ast  \alpha_i \ast \bar{w_i} = w_i \ast \bar{u_i} \ast u_i* \alpha_i \ast \bar{u_i} \ast u_i \ast \bar{w_i}
=  (w_i \ast \bar{u_i} ) \ast ( u_i *\alpha_i \ast \bar{u_i}) \ast \overline{(w_i \ast \bar{u_i})},
\]
where $w_i \ast \bar{u_i} \in \pi_1(X) \subset G(K)$ and $u_i \ast \alpha_i \ast \bar{u_i} \in \pi_1(E_i) \subset G(K)$. 
See Figure~\ref{splitting_graph_groups}. 

\begin{figure}[htb]
\centering
\includegraphics[width=0.65\textwidth]{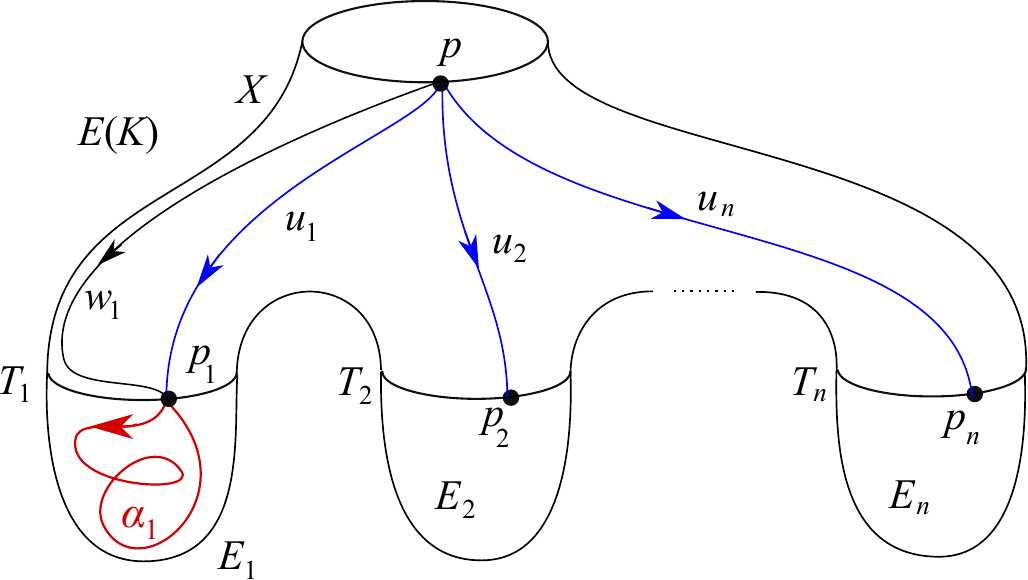}
\caption{The subgroups $\pi_1(E_i)$.} 
\label{splitting_graph_groups}
\end{figure}

\noindent
\subsection{Decomposition of elements of $G(K)$}
\label{decomposition_elements}

In this subsection, we will observe the following. 

\begin{lemma}
\label{express_g}
For a given element $g$ in $G(K)$, 
we express it as the product of elements $g_1, \dots, g_m$ so that the following two conditions hold: 

$(A)$ $g_i$ is non-trivial and belongs to $\pi_1(X)$ or $\pi_1(E_j)$ for some $1 \le j \le n$, and 

$(B)$ $g_i$ and $g_{i+1}$ do not belong to the same factor group.
\end{lemma}

\begin{proof}
Let $g$ be a non-trivial element of $G(K)$. 
Recall that we have chosen the  path $u_i$ from the base point $p$ to $p_i$ in $X$, 
which specifies the subgroups $\pi_1(E_i)$ in $\pi_1(E(K)) = G(K)$. 

Assume that $g \in G(K)$ is represented by an oriented closed curve $c$ based at $p \in T = \partial E(K)$. 
We may assume $c$ intersects each torus $T_i$ transversely. 
Traveling along $c$ according to its orientation, 
we assign a point $q_1, q_2, \dots, q_{k-1}, q_k$ when we cross $T_1 \cup \cdots \cup T_n$. 
Then these points decompose $c$ as $a_1 \cup b_1\cup a_2 \cup b_2 \cup a_3 \cup \cdots \cup a_k$, 
where $a_i \subset X$ and $b_i$ lies in exactly one of $E_1, \dots, E_n$; see Figure~\ref{elements_graph_groups}(Left). 
Choose a path $\tau_i$ arbitrarily on $T_1 \cup \cdots \cup T_n$ from  $p_{\ell}$ to $q_i$.     
Eventually we have a decomposition of $g$ as 
$g = g_1 g_2 \cdots g_s$, 
where each $g_i$ is non-trivial and belongs to $\pi_1(X)$ or $\pi_1(E_j)$ for some $1 \le j \le n$ 
as described in Example~\ref{decomposition_g} below.

\begin{example}
\label{decomposition_g}
Suppose that $E(K) = X \cup E_1 \cup\cdots \cup E_n$ and 
$g \in G(K)$ is represented by a closed curve $c$ as in Figure~\ref{elements_graph_groups}(Right).  
Then $g$ is expressed as 
\[
(a_1 \ast \overline{\tau_1} \ast \overline{u_1}) \ast (u_1 \ast  \tau_1 \ast b_1 \ast \overline{\tau_2} \ast \overline{u_1}) \ast 
(u_1 \ast \tau_2 \ast  a_2 \ast \overline{\tau_3} \ast \overline{u_2}) \ast (u_2 \ast \tau_3 \ast b_2 \ast \overline{\tau_4} \ast \overline{u_2}) \ast  (u_2 \ast \tau_4 \ast a_3), 
\]
where 
$a_1 \ast \overline{\tau_1} \ast \overline{u_1},\ 
u_1 \ast \tau_2 \ast  a_2 \ast \overline{\tau_3} \ast \overline{u_2}, \ 
u_2 \ast \tau_4 \ast a_3$ represent elements in $\pi_1(X)$, 
$u_1 \ast  \tau_1 \ast b_1 \ast \overline{\tau_2} \ast \overline{u_1}$ represents an element in $\pi_1(E_1)$, 
and
$u_2 \ast \tau_3 \ast b_2 \ast \overline{\tau_4} \ast \overline{u_2}$ represents an element in $\pi_1(E_2)$.  
\end{example}

\begin{figure}[htb]
\centering
\includegraphics[width=1.0\textwidth]{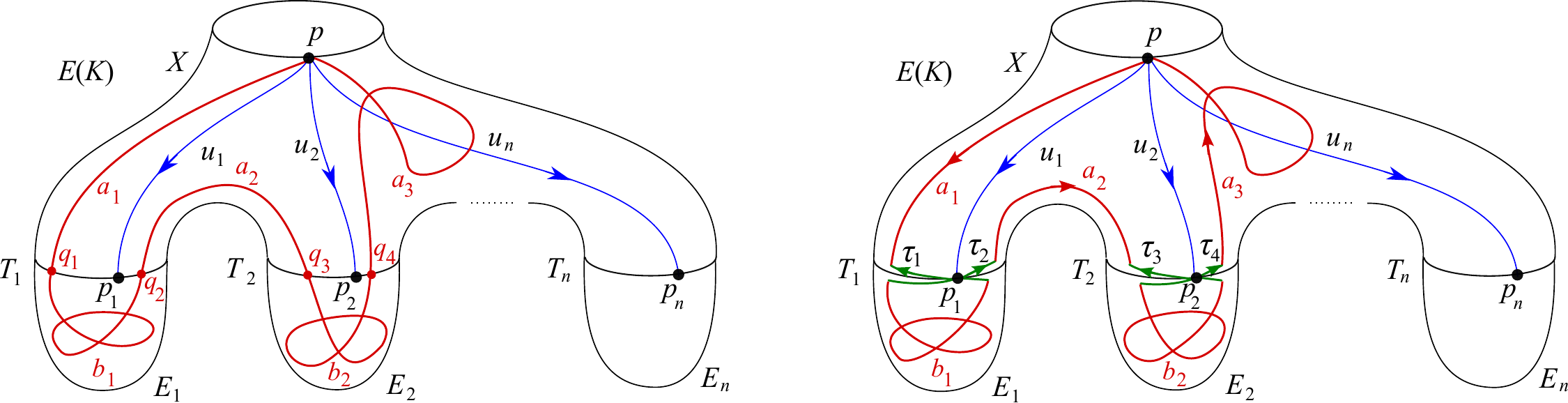}
\caption{Decomposition of elements as products of elements in factor groups.} 
\label{elements_graph_groups}
\end{figure}

\medskip

Finally if two successive elements $g_i$ and $g_{i+1}$ belong to the same factor group, 
then we combine them according to the following convention to get an expression satisfying conditions (A) and (B). 

\medskip

\noindent
\textbf{Convention.}
Suppose that $g_i$ belongs to $\pi_1(T_j)$ for some $j$. 
Then there are two possible choices:\  
$g_i \in \pi_1(X)$ or $g_i \in \pi_1(E_j)$.  
We regard $g_i \in \pi_1(X)$ or $g_i \in \pi_1(E_j)$ so that we successively combine as many $g_i$ as possible so that $g_i$ and $g_{i+1}$ do not belong to the same factor group. 
After this combining procedure, if we still have $g_i \in \pi_1(T_j)$ for some $j$, 
i.e. $g_{i-1} \in \pi_1(E_{j'}), g_{i+1} \in \pi_1(E_{j''})$ for some $j', j'' \ne j$, 
then we regard $g_i \in \pi_1(X)$ so that 
$g_i \in \pi_1(E_j)$ means that $g_i \in \pi_1(E_j) - \pi_1(T_j)$. 

\medskip 

Then we obtain a desired expression of $g \in G(K)$. 
\end{proof}

\medskip

We call an expression of $g$ satisfying conditions (A) and (B) a {\em minimal} expression of $g \in G(K)$. 
In the following, any expression of $g$ is always assumed to be minimal.  
Note that a minimal expression for the given element may not be unique. 

\medskip

\begin{example}
\begin{enumerate}
\item
Assume that $g_i \in \pi_1(X),\ g_{i+1} \in \pi_1(T_j)$ and $g_{i+2} \in \pi_1(X)$. 
Then combine them to obtain $g_ig_{i+1}g_{i+2} \in \pi_1(X)$. 
\item
Assume that 
$g_i \in \pi_1(E_j), g_{i+1} \in \pi_1(T_j), g_{i+2} \in \pi_1(T_{j'})$, and $g_{i+3} \in \pi_1(E_{j'})$ for $j' \ne j$. 
Then we combine 
$g_i g_{i+1} g_{i+2} g_{i+3} = (g_i g_{i+1})(g_{i+2} g_{i+3})$, 
where $g_i g_{i+1} \in \pi_1(E_j)$ and $g_{i+2} g_{i+3} \in \pi_1(E_{j'})$, 
or combine
$g_i g_{i+1} g_{i+2} g_{i+3} = g_i (g_{i+1}g_{i+2}) g_{i+3}$, 
where $g_i \in \pi_1(E_j), g_{i+1} g_{i+2}  \in \pi_1(X)$ and  $g_{i+3} \in \pi_1(E_{j'})$. 
\end{enumerate}
\end{example}

\medskip

\begin{lemma}
\label{regular_fiber}
Assume that $X$ is Seifert fibered, 
i.e. it is either a composing space or a cable space.
Let $g=g_1\cdots g_m$ be a minimal expression. If $m \ge 2$, and
$g_{i} \in \pi_1(X)$ for some $i$, then $g_i$ is not represented by a regular fiber.  
\end{lemma}

\begin{proof}
Suppose that $g_{i} \in \pi_1(X)$ is represented by a regular fiber of $X$. 
Then $g_{i} \in \pi_1(T_j) \subset \pi_1(E_j)$ for $j =1, 2,\dots, n$ (when $X$ is a composing space) and $g_i \in \pi_1(T_1) \subset \pi_1(E_1)$ (when $X$ is a cable space). 

Hence, the regular fiber belongs to every factor group, so one would
find a contradiction with condition (B) if $g_i$ is represented by a regular fiber.
\end{proof}

\bigskip

\section{Dehn filling trivialization}
\label{Dehn_filling}
Let $(\mu, \lambda)$ be a preferred meridian-longitude pair of $K$, which generate $\pi_1(T)$. 
Then every simple closed curve on $T$ represents $\mu^p \lambda^q \in \pi_1(T)$ for some coprime integers $p$ and $q$.  
Denote by $X(p/q)$ the $3$--manifold obtained from $X$ by
$p/q$--Dehn filling on $T$, namely, the 3-manifold obtained by 
 attaching a solid torus $V$ to $X$ along the boundary component $T = \partial E(K)$ so that the meridian of $V$ represents a slope element $\mu^p \lambda^q$. 

For $r=p/q$,
$r$--Dehn filling of $X$ induces an epimorphism 
\[
p_{X, r} \colon \pi_1(X) \to  \pi_1(X(p/q)) = \pi_1(X) / \langle\!\langle  \mu^p \lambda^q \rangle\!\rangle, 
\]
where $\langle\!\langle  \mu^p \lambda^q \rangle\!\rangle$ denotes the normal closure of $\mu^p \lambda^q$ in $\pi_1(X)$. 
Thus, $p/q$--Dehn filling trivializes elements in $\langle\!\langle  \mu^p \lambda^q \rangle\!\rangle \subset \pi_1(X)$. 
(In case of $X = E(K)$, we often abbreviate $p_{X, r}$ as $p_r$.) 

 In the context of knot group, Theorem~\ref{1/n} has the following paraphrase.  

\begin{remark}
\label{rephrase}
Let $K$ be a non-trivial knot. 
Then  
$\displaystyle\bigcap_{n \in I} \langle\!\langle 1/n \rangle\!\rangle = \{ 1 \}$ for any infinite subset $I$ of $\mathbb{Z}$. 
In particular, 
$\displaystyle\bigcap_{r \in \mathbb{Q}} \langle\!\langle r \rangle\!\rangle = \{ 1 \}$. 
\end{remark}

Now we assume that $X \ne E(K)$. 
If $X$ is Seifert fibered, 
then it is either a cable space or a composing space \cite{JS}. 
Correspondingly, $K$ is a cable knot or a composite knot.   
Denote $\partial X = T \cup T_1 \cup \cdots \cup T_n$. 
Note that each $T_i$ bounds a non-trivial knot exterior $E_i$ in $E(K)$ for $1 \le i \le n$. 
Then $E(K)$ is expressed as  $E(K) = X \cup E_1 \cup \cdots \cup E_n$. 
Recall that we have chosen and fixed a base-point $p_i \in T_i$ and a path $u_i$ in $X$ from $p$ to $p_i$ to define the subgroup $\pi_1(T_i),\pi_1(E_i)$. 
We classify non-trivial elements $g \in \pi_1(X)$ into three classes: 

\begin{itemize}
\item
$g \in \pi_1(X)$ is {\em outer-peripheral} if $g$ is conjugate into $\pi_1(T) = \pi_1(\partial E(K))$. 
An outer-peripheral element is conjugate to a power of some slope element. 

\item
$g \in \pi_1(X)$ is {\em inner-peripheral} 
if it is conjugate to an element of $\pi_1(T_i) \subset \pi_1(X)$ for some $i=1,\ldots, n$.  
 
\item
$g \in \pi_1(X)$ is {\em non-peripheral} if it is neither outer-peripheral nor inner-peripheral. 
\end{itemize}

We remark that the definition of the subgroup $\pi_1(T_i) \subset \pi_1(X)$ depends on a choice of a path $u_i$, but as we have seen, 
different paths yield conjugate subgroups, so these notions do not depend on the choice of path and are well defined. 

If $X = E(K)$, then there is no inner-peripheral element, and an outer-peripheral element is simply called a peripheral element. 

\medskip

Let $g$ be an outer-peripheral element of $\pi_1(X)$. 
Then $g = \alpha (\mu^a \lambda^b)^k \alpha^{-1}$ in $\pi_1(X)$ for some element $\alpha \in \pi_1(X)$ and integers $a, b$ and $k$. 
We call $g$ an outer-peripheral element of slope $a/b \in \mathbb{Q}\cup \{1/0\}$. 

\medskip

\subsection{Dehn fillings of a hyperbolic manifold}
\label{outer_piece_hyperbolic}

\begin{proposition}
\label{filling_X_hyperbolic}
Suppose that $X$ is hyperbolic. 
\begin{enumerate}
\item
Except for at most two non-trivial slopes $r \in \mathbb{Q}$, 
$X(r)$ is boundary irreducible, and hence $\pi_1(X(r))$, $\pi_1(E_j)$ inject into $\pi_1(K(r))$. 
\item
If $g \in \pi_1(T_j)$ for some $1 \le j \le n$, 
then except for at most two slopes $r \in \mathbb{Q}$, $1 \ne p_{X, r}(g)\in \pi_1(T_j) \subset \pi_1(X(r))$. 
\item  
If $g \in \pi_1(X) - \cup_{j=1}^n\pi_1(T_j)$, 
then $p_{X, r}(g) \in \pi_1(X(r)) - \cup_{j=1}^n\pi_1(T_j)$ for all but finitely many slopes $r \in \mathbb{Q}$. 
\end{enumerate}
\end{proposition}

Before proving this proposition, 
for convenience of readers, we just briefly recall Thurston's hyperbolic Dehn surgery theory \cite{T1,T2}. 
By the assumption we have a holonomy (faithful and discrete) representation
\[
\rho \colon \pi_1(X) \to \mathrm{Isom}^+(\mathbb{H}^3) = \mathrm{PSL}(2, \mathbb{C}).
\]

We regard $\rho = \rho_{\infty}$ and denote its image by $\Gamma_{\infty} \cong \pi_1(X)$. 
Recall that $\rho(\mu)$ and $\rho(\lambda)$ are parabolic elements.  
By small deformation of the holonomy representation $\rho$ which keeps $\partial X - T$ as cusps (up to conjugation), 
we obtain representations 
$\rho_{x, y} \colon \pi_1(X) \to \mathrm{Isom}^+(\mathbb{H}^3) = \mathrm{PSL}(2, \mathbb{C})$ whose image is 
$\Gamma_{x, y}$. 

Deforming $\rho$ continuously, 
$(x, y)$ varies over an open set $U_{\infty}$ in $S^2 = \mathbb{R}^2 \cup \{ \infty \}$; see the proof of \cite[Theorem~5.8.2]{T1}. 
Note that under deformation, the trace of $\rho_{x, y}$ is also deformed continuously. 
By Mostow-Prasad rigidity theorem \cite{Mostow,Prasad} $\rho_{x, y}(\mu)$ and $\rho_{x, y}(\lambda)$ are not parabolic, i.e. loxodromic when $(x, y) \ne \infty$. 
Note that for any element $\gamma \in \pi_1(T_j)$, 
$\rho(\gamma)$ is parabolic and $\rho_{x, y}(\gamma)$ remains parabolic for all $1 \le j \le n$. 
When $(x, y)$ is $(p, q)$ for coprime integers $p$ and $q$ with $|p| + |q|$ sufficiently large, 
$\rho_{p, q}$ gives an incomplete hyperbolic metric of $\mathrm{int}X$ so that 
its completion is $X(p/q)$ in which $K_{p/q}^*$, the image of the filled solid torus, 
is the shortest closed geodesic in the hyperbolic $3$--manifold $X(p/q)$. 
For the non-faithful representation $\rho_{p, q}$ of $\pi_1(X)$, 
its image $\Gamma_{p, q}$ is regarded as $\pi_1(X(p/q))$. 
\[
\xymatrix{
\pi_1(X)\ar@{->>}[drr]_{\rho_{p,q}}  \ar[rr]^{\rho}_{\cong} &  & \hspace{0.2cm} \Gamma\subset \mathrm{Isom}^+(\mathbf{H}^3)\ar@{->>}[d] & \\
              &  &  \hspace{1.0cm}  \Gamma_{p,q}\subset \mathrm{Isom}^+(\mathbf{H}^3)
}
\]

Note that for an outer-peripheral element $\gamma$, 
$\rho(\gamma)$ is parabolic, while $\rho_{p, q}(\gamma)$ is loxodromic for $p/q$ with $|p| + |q|$ sufficiently large. 
If $\delta$ is inner-peripheral, then $\rho(\delta)$ and $\rho_{p, q}(\delta)$ are both parabolic. 
Let $g$ be a non-peripheral element of $\pi_1(X)$, 
then $\rho(g)$ and $\rho_{p, q}(g)$ are loxodromic. 

\begin{proof}[Proof of Proposition~\ref{filling_X_hyperbolic}]
$(1)$ Following \cite[Theorems~2.4.4 and 2.4.5]{CGLS}, $X(r)$ is boundary-reducible for at most two slopes $r \in \mathbb{Q}$. 

$(2)$ Assume that $g$ belongs to $\pi_1(T_j)$ for some $1 \le j \le n$. 
Since $X(r)$ is boundary-irreducible except for at most two non-trivial slopes, 
$p_{X, r}(g)$ is non-trivial in $\pi_1(X(r))$ except for at most two non-trivial slopes. 

$(3)$ 
First we assume that $g$ is outer-peripheral. 
If $p_{X, r}(g) = 1$ in $\pi_1(X(r))$, 
then $p_r(g) = 1 \in \pi_1(K(r))$. 
Following \cite[Proposition~3.1]{IMT_realization}, 
$r$ is a finite surgery slope or $r$ is the slope of $g$. 
Since $K$ has only finitely many finite surgery slopes, 
$p_{X,r}(g) \ne 1$ for all but finitely many slopes $r \in \mathbb{Q}$. 
Furthermore, 
following the above observation, 
$g$ becomes a loxodromic element except for only finitely many slopes. 
Hence, 
$p_{X,r}(g)$ is non-trivial and non-peripheral in $\pi_1(X(r))$ for all but finitely many slopes $r \in \mathbb{Q}$. 

\smallskip

Suppose that $g$ is inner-peripheral, but $g \not\in \pi_1(T_j)$ for $1 \le j \le n$. 

By the assumption we may write 
$g = \alpha \gamma \alpha^{-1}$ , where $\alpha \in \pi_1(X)$ and $\gamma \in \pi_1(T_j)$ for some $1 \le j \le n$. 
Since $g \not\in \pi_1(T_j)$, $\alpha \not\in \pi_1(T_j)$. 
We will show that $p_{X, r}(g) \not \in \pi_1(T_j)$ $(1 \le j \le n)$ for all but finitely many slopes $r$. 
Divide our argument into three cases according as 
(i) $\alpha$ is non-peripheral, 
(ii) $\alpha$ is outer-peripheral, or 
(iii) $\alpha$ is inner-peripheral. 

Let us take a holonomy representation $\rho \colon \pi_1(X) \to \mathrm{Isom}^+(\mathbb{H}^3) = \mathrm{PSL}(2, \mathbb{C})$. 

(i) $\alpha$ is non-peripheral. 
Then $\rho(\alpha)$ is loxodromic, $\rho(\gamma)$ is parabolic,  
and after small deformation, $\rho_{x, y}(\alpha)$ remains loxodromic and $\rho_{x, y}(\gamma)$ remains parabolic. 
Hence, $\rho_{x, y}(\alpha)$ and $\rho_{x, y}(\gamma)$ do not commute and $p_{X, r}(\alpha) \in \pi_1(X(r))$ does not belong to $\pi_1(T_j)$. 
Since $\pi_1(T_j)$ is malnormal in $\pi_1(X(r))$ \cite[Theorem~3]{HW}, 
$p_{X, r}(g) = p_{X, r}(\alpha) p_{X, r}(\gamma) p_{X, r}(\alpha)^{-1} \not\in \pi_1(T_j)$.

(ii) $\alpha$ is outer-peripheral. 
Then $\rho(\alpha)$ is parabolic, $\rho(\gamma)$ is parabolic,  
and after small deformation, $\rho_{x, y}(\alpha)$ becomes loxodromic, while $\rho_{x, y}(\gamma)$ remains parabolic. 
Hence, $\rho_{x, y}(\alpha)$ and $\rho_{x, y}(\gamma)$ do not commute and $p_{X, r}(\alpha) \in \pi_1(X(r))$ does not belong to $\pi_1(T_j)$. 
As in (i) we see that 
$p_{X, r}(g) = p_{X, r}(\alpha) p_{X, r}(\gamma) p_{X, r}(\alpha)^{-1} \not\in \pi_1(T_j)$. 

(iii) $\alpha$ is inner-peripheral. 
Then $\rho(\alpha)$ and $\rho(\gamma)$ are both parabolic, and 
after small deformation $\rho_{x, y}(\alpha)$ and $\rho_{x, y}(\gamma)$ remain parabolic. 
Since $\gamma \in \pi_1(T_j)$ and $g = \alpha \gamma \alpha^{-1} \not\in \pi_1(T_j)$,  
we have $\alpha \not\in \pi_1(T_j)$. 
Hence $\mathrm{Fix}(\rho(\alpha)) \ne \mathrm{Fix}(\rho(\gamma))$, 
where $\mathrm{Fix}(\cdot)$ denotes the fixed point set of $\cdot$ in $S^2_{\infty}$, the sphere at infinity.
After small deformation, we have still $\mathrm{Fix}(\rho_{x, y}(\alpha)) \ne \mathrm{Fix}(\rho_{x, y}(\gamma))$. 
This means that $p_{X, r}(\alpha) \not\in \pi_1(T_j)$, while $p_{X, r}(\gamma) \in \pi_1(T_j)$. 
So malnormality of $\pi_1(T_j)$ in $\pi_1(X)$ shows that $p_{X, r}(\alpha) p_{X, r}(\gamma) p_{X, r}(\alpha)^{-1} \not\in \pi_1(T_j)$. 

Furthermore, in either case, 
since $p_{X, r}(\alpha) p_{X, r}(\gamma) p_{X, r}(\alpha)^{-1}$ is conjugate to 
$p_{X, r}(\gamma) \in \pi_1(T_j)$,   
we see that $p_{X, r}(\alpha) p_{X, r}(\gamma) p_{X, r}(\alpha)^{-1} \not\in \pi_1(T_{j'})$ for $j' \ne j$. 
This shows that  $p_{X, r}(\alpha) p_{X, r}(\gamma) p_{X, r}(\alpha)^{-1} \not\in \pi_1(T_{j})$ for all $j = 1, \dots, n$. 

\smallskip

Finally we  assume that $g$ is a non-peripheral element in $\pi_1(X)$, 
and prove that it remains non-trivial and non-peripheral in $\pi_1(X(r))$ for all but finitely many slopes $r$. 

Since $g$ is non-peripheral, $\rho(g)$ is loxodromic, and $\mathrm{Tr}(\rho(g)) \ne \pm 2$.

Under small deformation of $\rho$, 
$\mathrm{Tr}(\rho_{x, y}(g))$ deforms continuously according to $(x, y) \in U_{\infty}$. 
Hence for $p/q$ with $|p| + |q|$ sufficiently large, 
$\mathrm{Tr}(\rho_{p, q}(g)) \ne \pm 2$. 
In particular, $\rho_{p, q}(g)$ is non-parabolic. 
Hence $p_{p/q}(g)$ remains non-trivial and non-peripheral for all but finitely many slopes. 
\end{proof}

\begin{remark}
\label{K_hyperbolic}
If $K$ is a hyperbolic knot, then $X=E(K)$ and $T_j = \emptyset$ $(1 \le j \le n)$, 
and Proposition~\ref{filling_X_hyperbolic} (3) shows that every non-trivial element of $G(K)$ remains non-trivial in $\pi_1(K(r))$ for all but finitely many slopes $r \in \mathbb{Q}$. 
\end{remark}

\medskip

In the following subsections~\ref{outer_piece_composing}, \ref{outer_piece_cable} and \ref{torus_knot}, 
for our purpose we restrict our attention to $1/n$--Dehn fillings. 

\subsection{Dehn fillings of a composing space}
\label{outer_piece_composing}

We say that $X$ is a $k$--fold composing space if 
it is homeomorphic to $[\textrm{disk with $k$ holes}] \times S^1$, 
which arise in the torus decomposition of the exterior of a composite knot $K = K_1\; \#\; K_2\; \# \cdots \#\; K_k$, 
where each $K_i$ is prime. 
For our purpose it is convenient to regard $K = K_1\; \#\; K'_2$, where $K'_2 = K_2\; \# \cdots \#\; K_k$, and correspondingly decompose $E(K) = X \cup E(K_1) \cup E(K'_2)$, where $X$ is the $2$--fold composing space.  
Note that this does not give the torus decomposition. 
For non-prime knots $K$, in the following we consider such a decomposition of $E(K)$. 

\begin{proposition}
\label{filling_X_composing}
Suppose that $X$ is the $2$--fold composing space with $\partial X = T \cup T_1 \cup T_2$, 
where $T = \partial E(K)$.  
Assume that $n$ is a positive integer. 
\begin{enumerate}
\item
$X(1/n)$ is boundary irreducible, and hence $\pi_1(X(1/n))$, 
$\pi_1(E_j)$ inject into $\pi_1(K(1/n))$. 
\item
If $g \in \pi_1(T_j)$ for some $1 \le j \le 2$, 
then $p_{X, r}(g)$ is non-trivial in $\pi_1(T_j) \subset \pi_1(X(1/n))$. 
\item
If $g \in \pi_1(X) -\pi_1(T_1) \cup \pi_1(T_2)$, 
then there exists a constant $N_g > 0$ such that 
$p_{X, 1/n}(g) \in \pi_1(X(1/n)) - \left(\pi_1(T_1) \cup \pi_1(T_2)\right)$ for all $n \ge N_g$.  
\end{enumerate}
\end{proposition}

\begin{proof}
Recall that in the (product) Seifert fibration of $X$, 
the meridian $\mu$ of $K$ is a regular fiber. 
Thus $X(1/n)$ is a Seifert fiber space over the annulus with at most one exceptional fiber of index $n$ which is the dual of $K$ (the image of the core of the attached solid torus). 
Precisely, $X(1/n)$ is $S^1 \times S^1 \times [0, 1]$ if $n = 1$, 
otherwise $X(1/n)$ is a Seifert fiber space over the annulus with a single exceptional fiber $K^*$ of index $n \ge 2$. 

(1) The above observation shows that $X(1/n)$ is boundary-irreducible. 

\medskip

(2) Assume that $g$ belongs to $\pi_1(T_j) \subset \pi_1(X)$. 
Following (1), $p_{X, 1/n}(g)$ is non-trivial in $\pi_1(X(1/n))$. 

\medskip

(3) 
Let us consider the following presentation of $\pi_1(X)$: 
\[
\langle c, d, t \mid [c, t]= [d, t] = 1 \rangle \cong (\mathbb{Z} \ast \mathbb{Z}) \oplus \mathbb{Z}; 
\]
$c = \lambda \in \pi_1(\partial E(K))$, $d \in \pi_1(T_1)$ and $cd \in \pi_1(T_2)$,  
and $t = \mu \in \pi_1(\partial E(K))$ is represented by a regular fiber.  
Here we choose a specific conjugacy representatives, i.e. a specific path $u_i$ introduced in Section~\ref{fundamental_groups} to obtain the above presentation.

Assume first that $g \in \pi_1(T) = \pi_1(\partial E(K))$. 
Then $g$ is written as $\mu^a \lambda^b = t^a c^b$ for some integers $a, b$ which may not be coprime. 
If $b = 0$, then $g = t^a$ belongs to $\pi_1(T_i)$ $(i = 1, 2)$, contradicting the assumption. 
So in the following we assume that $b \ne 0$.

Since $t = \mu$ and $c = \lambda$, 
in $X(1/n)$, a meridian of the filled solid torus represents $t c^n$, 
and  
$\pi_1(X(1/n))$ has a presentation:
\[
\langle c, d, t \mid [c, t]= [d, t] = 1,\  tc^{n} = 1 \rangle. 
\]

Taking the quotient by the normal subgroup $\langle t \rangle$, 
we have a projection 
\[
\pi \colon \pi_1(X(1/n)) \to \pi_1(X(1/n))/\langle t \rangle,
\] 
where $\pi_1(X(1/n))/\langle t \rangle$ has the presentation: 
\[
\langle c, d \mid  c^{n} = 1 \rangle \cong \langle c \mid c^{n} = 1 \rangle \ast \langle  d \rangle \cong  \mathbb{Z}_{n} \ast \mathbb{Z}. 
\]

Now let us take $n$ so that $n > |b|$.  
Then $\pi(p_{1/n}(g)) = c^b$ is non-trivial and cannot be any power of $d$ or $cd$ in $\pi_1(X(1/n))/\langle t \rangle$. 
This means that $p_{1/n}(g) \not\in \pi_1(T_i)$ for $i = 1, 2$.

\smallskip

Let us suppose that $g \not\in \pi_1(T)$. 
Thus $g \in \pi_1(X) - \left(\pi_1(T) \cup \pi_1(T_1) \cup \pi_1(T_2)\right)$ 
by the assumption. 
We show that there exists a constant $N _g > 1$ such that $p_{X, 1/n}(g)$ is non-trivial and $p_{X, 1/n}(g) \in \pi_1(X(1/n)) - \left(\pi_1(T_1)\cup \pi_1(T_2)\right)$
for all $n \ge N_g$. 

Since $t$ is central, 
$g \in \pi_1(X)$ is written as $t^k w(c, d)$ for some integer $k$ and for some word $w(c, d)$ of $c$ and $d$. 
Furthermore, since $g \in \pi_1(X) - \left(\pi_1(T) \cup \pi_1(T_1)\cup \pi_1(T_2)\right)$, 
the word $w(c, d)$ is non-trivial, 
and both $c$ and $d$ appear in $w(c, d)$, 
and $g$ is not a power of $c$, $d$ or $cd$. 
(Note that $t$ belongs to $\pi_1(T), \pi_1(T_1)$ and $\pi_1(T_2)$.)
Let us continue to write $p_{X,1/n}(g) = t^k w(c, d) \in \pi_1(X(1/n))$.
To see that it remains non-trivial and non-peripheral, 
we look at $\pi(p_{X,1/n}(g)) = \pi(w(c, d))$ in $\pi_1(X(1/n))/\langle t \rangle$.

Recall that in the given reduced word $w(c, d)$, 
both $c$ and $d$ appear and $w(c, d)$ is none of a power of $c, d$ or $cd$. 
Write $w(c, d) = c^{p_1} d^{q_1} \cdots c^{p_k} d^{q_k} c^{p_{k+1}}$ for some integers $p_1, q_1 \dots p_{k+1}$, 
where $p_1$, $p_{k+1}$ may be $0$, while others are not $0$. 
 
Set $p_c = \mathrm{max}\{ |p_1|, \dots , | p_k |, | p_{k+1}| \}$ and 
take $n$ so that $n > p_c + 1$. 
Then $c^{p_i} \ne 1 \in \mathbb{Z}_n$ whenever $p_i \ne 0$. 
Hence, $\pi(w(c, d))$ is non-trivial and not a power of $d$ in $\pi_1(X(1/n))/\langle t \rangle$, 
which implies that $p_{X, 1/n}(g) \ne 1 \in \pi_1(X(1/n))$ and $p_{X, 1/n}(g) \not\in \pi_1(T_1)$. 
Let us see if $\pi(w(c, d))$ is a power of $cd$ in $\pi_1(X(1/n))/\langle t \rangle$. 
If $\pi(w(c, d))$ is a power of $cd$, then it is either $c^{-p}d c^{-p}d \cdots c^{-p}d$ or $d^{-1}c^p d^{-1}c^p \cdots d^{-1}c^p$ and $n = p+1$. 
In the former case $\pi(w(c, d)) = cd \cdots cd \in \mathbb{Z}_{p+1} \ast \mathbb{Z}$ 
and in the latter case $\pi(w(c, d)) = d^{-1}c^{-1}\cdots d^{-1}c^{-1} \in \mathbb{Z}_{p+1} \ast \mathbb{Z}$, which are powers of $cd$. 
However, since $n > p_c + 1$, such situations cannot happen. 
(This situation was pointed by the referee.)

Hence, $p_{X, 1/n}(g)$ is non-trivial and not contained in $\pi_1(T_j)$ if $n \ge N_g = \mathrm{max}\{|b| + 1,\  p_c + 2\}$. 
\end{proof}

\medskip

\subsection{Dehn fillings of a cable space}
\label{outer_piece_cable}

We say that $K$ is a $(p, q)$--cable of a non-trivial knot $k$ if $K$ is embedded in the boundary of a small tubular neighborhood of $k$ such that it wraps $p$ time in meridional direction and $q$ times in longitudinal direction; for non-triviality we assume $q \ge 2$. 

\begin{proposition}
\label{filling_X_cable}
Suppose that $X$ is a $(p, q)$--cable space with $\partial X = T \cup T_1$, 
where $T = \partial E(K)$.  
Assume that $n$ is an integer greater than $1$. 
\begin{enumerate}
\item
$X(1/n)$ is boundary irreducible, and hence $\pi_1(X(1/n))$, $\pi_1(E_j)$ inject into $\pi_1(K(1/n))$. 
\item
If $g \in \pi_1(T_1)$, 
then $p_{X, 1/n}(g)$ is non-trivial in $\pi_1(T_1) \subset \pi_1(X(1/n))$. 
\item
If $g \in \pi_1(X) -\pi_1(T_1)$, 
then 
there exists a constant $N_g > 0$ such that 
$p_{X, 1/n}(g) \in \pi_1(X(1/n)) - \pi_1(T_1)$ for all $n$ with $n \ge N_g$.  
\end{enumerate}
\end{proposition}

\begin{proof}
Since $n \ge 2$, $|pqn-1| \ge 3  > 2$. 
Then, after $1/n$--Dehn filling,  
$X(1/n)$ is a Seifert fiber space over the disk with two exceptional fibers of indices $q$ and $|pqn-1| > 2$. 

$(1)$ 
The above observation shows that  $X(1/n)$ is boundary-irreducible. 

\medskip

$(2)$ Assume that $g \in \pi_1(T_1)$. 
Following (1),  
$g$ remains non-trivial in $\pi_1(X(1/n))$. 

\medskip

$(3)$ 
Using a Seifert structure of $X$, we have the following presentation of $\pi_1(X)$. 
\[
\langle c_1, c_2, t \mid [c_1, t]= [c_2, t] = 1,\ c_2^qt^\alpha = 1 \rangle,  
\]
where $t = \mu^{pq} \lambda \in \pi_1(T) = \pi_1(\partial E(K))$ is represented by a regular fiber, which belongs to $\pi_1(T_1)$ as well, 
$c_1 \in \pi_1(T)$ and $t$ generate $\pi_1(T)$, and $c_1c_2 \in \pi_1(T_1)$ and $t$ generate $\pi_1(T_1)$. 
Here we choose a specific conjugacy representatives, i.e. a specific path $u_i$ introduced in Section~\ref{fundamental_groups} to obtain the above presentation.

\medskip

Suppose that $g \in \pi_1(T)$. 
Then $g$ is expressed as $c_1^a t^b$ for some integers $a, b$, which may not be coprime. 
If $a = 0$, then $g = t^b \in \pi_1(T_1)$, contradicting the assumption. 
So in the following, we assume that $a \ne 0$. 

To get a presentation of $\pi_1(X(1/n))$, 
write $c_1 = \mu^x\lambda^y$ for some coprime integers $x, y$ such that $|pqy - x| =1$. 
Then, $1/n$--Dehn filling gives a relation $\mu \lambda^n$, which can be re-written by $c_1^{pqn-1} t^\beta = 1$ for some integer $\beta$. 
Thus we obtain the following presentation of $\pi_1(X(1/n))$: 
\[
\langle c_1, c_2, t \mid [c_1, t]= [c_2, t] = 1,\ c_2^qt^\alpha = 1,\ c_1^{pqn-1} t^\beta = 1 \rangle. 
\]

Taking the quotient by the normal subgroup $\langle t \rangle$, 
we have a projection 
\[
\pi \colon \pi_1(X(1/n)) \to \pi_1(X(1/n))/\langle t \rangle,
\] 
where $\pi_1(X(1/n))/\langle t \rangle$ has the presentation: 
\[
\langle c_1, c_2 \mid c_2^q = 1, c_1^{|pqn-1|} = 1 \rangle = \mathbb{Z}_q \ast \mathbb{Z}_{|pqn-1|}. 
\]

Now let us take $n$ so that $|pqn -1 | > |a|$. 

If $p_{X, 1/n}(g) \in \pi_1(T_1)$, 
then $\pi(p_{X, 1/n}(g)) = \pi(c_1^a t^b) = c_1^a$ is non-trivial and cannot be a power of $c_1c_2$ in  
$\pi_1(X(1/n))/\langle t \rangle$. 
This means that $p_{X, 1/n}(g) \not\in \pi_1(T_1)$. 

\medskip

Assume that $g \not\in \pi_1(T)$. 
Then $g \in \pi_1(X) - \left(\pi_1(T) \cup \pi_1(T_1)\right)$. 

For a given $g \in \pi_1(X)$, 
since $t$ is central, 
we may write $g = t^k w(c_1, c_2)$ for some integer $k$ and for some word $w(c_1, c_2)$ of $c_1$ and $c_2$. 
Further, we write $w(c_1, c_2) = c_1^{p_1} c_2^{q_1} \cdots c_1^{p_{\ell}} c_2^{q_{\ell}} c_1^{p_{\ell+1}} \in \pi_1(X)$ for some integers $p_1, q_1 \dots p_{\ell}, q_{\ell}, p_{\ell+1}$, 
where $p_1$, $p_{\ell+1}$ may be $0$, while others are not $0$. 
Set $N_{g, c_1} = \mathrm{max}\{ |p_1|, \dots , |p_{\ell}|, |p_{\ell+1}| \}$, 
and $N_{g, c_2} = \mathrm{max}\{ |q_1|, \dots , |q_{\ell}| \}$. 
Applying the relations $[c_1, t]= [c_2, t] = 1$ and $c_2^q = t^{-\alpha}$, 
we may assume $N_{g, c_2} < q$. 

In the expression $g = t^k w(c_1, c_2)
= t^k c_1^{p_1} c_2^{q_1} \cdots c_1^{p_{\ell}} c_2^{q_{\ell}} c_1^{ p_{\ell+1}} \in \pi_1(X)$, 
since $g \in \pi_1(X) - \left(\pi_1(T) \cup \pi_1(T_1)\right)$, 
the word $w(c_1, c_2)$ is non-trivial, and not a power of $c_1$ or $c_1c_2$. 
Note that $c_2$ appears in $w(c_1, c_2)$, and 
as noted above, 
we may also assume $|q_i| < q$.  

Let us show that $p_{X,1/n}(g)$ is non-trivial and 
not contained in $\pi_1(T_1)$ provided $|pqn - 1| > N_{g, c_1} + 1$.  
We continue to write $p_{X,1/n}(g) = t^k w(c_1, c_2) \in \pi_1(X(1/n))$ and 
consider $\pi(p_{X,1/n}(g)) = \pi(w(c_1, c_2)) \in \pi_1(X(1/n))/\langle t \rangle$. 

Since $|p_i| < N_{g, c_1} + 1 < |pqn - 1|$, both $c_1$ and $c_2$ appear in 
$\pi(w(c_1, c_2)) \in \mathbb{Z}_q \ast \mathbb{Z}_{|pqn-1|}$, 
thus $p_{X,1/n}(g)$ is non-trivial and not contained in $\pi_1(T)$. 

It remains to see that $p_{X,1/n}(g) \not\in \pi_1(T_1)$. 
Suppose for a contradiction that $p_{X,1/n}(g) \in \pi_1(T_1)$. 
Then $\pi(w(c_1, c_2))$ is a power of $c_1c_2$ in $\pi_1(X(1/n)) / \langle t \rangle$. 
Recall that $w(c_1, c_2)$ is not a power of $c_1c_2$ in $\pi_1(X)$. 
So  this may occur only if $w(c_1, c_2) = c_1^{-N}c_2 c_1^{-N}c_2 \cdots c_1^{-N}c_2$ or 
$c_2^{-1}c_1^N c_2^{-1}c_1^N \cdots c_2^{-1}c_1^N$ and $|pqn-1| = N + 1$. 
In the former case $\pi(w(c_1, c_2)) = c_1c_2 \cdots c_1c_2 \in \mathbb{Z}_q \ast \mathbb{Z}_{|pqn-1|}$, 
and in the latter case $\pi(w(c_1, c_2)) = c_2^{-1}c_1^{-1}  \cdots c_2^{-1}c_1^{-1}  \in \mathbb{Z}_q \ast \mathbb{Z}_{|pqn-1|}$, 
which are powers of $c_1c_2$. 
However, since $n$ satisfies $|pqn-1| > N_{g, c_1} + 1$, this is impossible.  
Let us take a constant $N_g >0$ so that if $n \ge N_g$, then $|pqn-1| > |a|$ and $|pqn-1| > N_{g, c_1}$. 
Then $N_g$ is a desired constant. 
\end{proof}

\subsection{Dehn filling of a torus knot space}
\label{torus_knot}

\begin{proposition}
\label{torus _knot_fully}
Let $K$ be a non-trivial torus knot. 
Let $g$ be any non-trivial element of $G(K)$. 
Then there exists a constant $N_g > 0$ such that $p_{1/n}(g) \ne 1$ in $\pi_1(K(1/n))$ for all $n \ge N_g$. 
\end{proposition}

\begin{proof} 
Let $K$ be the torus knot $T_{p, q}$, 
where we assume $2 \le p < q$. The case where $q$ is negative is treated similarly. 
The exterior $E(T_{p,q})$ is a Seifert fiber space over $D^{2}$ with two exceptional fibers of indices $2\le p <q$, and its $1/n$-filling $T_{p, q}(1/n)$ is a Seifert fiber space over $S^2$ with three exceptional fibers of indices $2 \le p < q < pqn -1$.  

Using Seifert structures, 
we have the following presentation of $\pi_1(E(T_{p, q}))$: 
\[
\langle c_1, c_2, t \mid [c_1, t]= [c_2, t] = 1, c_1^p t^\alpha = c_2^q t^{\beta} = 1 \rangle,  
\]
for some integers $\alpha$ and $\beta$, where $t = \mu^{pq} \lambda \in \pi_1(\partial E(K))$ is represented by a regular fiber and 
$c_1c_2 \in \pi_1(\partial E(K))$. 
Similarly, 
we obtain a presentation of $\pi_1(K(r))$ by adding an extra relation corresponding to $1/n$--Dehn filling. 
\[
\langle c_1, c_2, t \mid [c_1, t]= [c_2, t] = 1, c_1^p t^\alpha = c_2^q t^{\beta} = (c_1c_2)^{pqn-1} t^{\gamma} = 1 \rangle. 
\]
for some integer $\gamma$. 
The last relation arises from $1/n$--Dehn filling. 

If $n>1$, then $pqn-1\geq 11>6$, so $\frac{1}{p}+\frac{1}{q}+\frac{1}{pqn-1}<1$.
Then $\pi_1(T_{p, q}(1/n))$ is infinite. 
 In particular, an element $p_{1/n}(t) \in \pi_1(T_{p, q}(1/n))$, 
denoted by the same symbol $t$, 
represented by the image of a regular fiber has infinite order. 

In the following, we always assume that $n>1$ so that we have the following commutative diagram. 
\[ 
\xymatrix{ 
1 \ar[r]& \mathbb{Z}=\langle t \rangle \ar[r] \ar[d]_{p_{1/n}}^{\cong}& G(T_{p,q}) \ar[r]^-{f} \ar[d]_{p_{1/n}} & \mathbb{Z}_p \ast \mathbb{Z}_q =\langle c_1, c_2 \: | \: c_1^p=c_2^q = 1 \rangle \ar[r] \ar[d]^{P_{1/n}} & 1 \\
1 \ar[r]& \mathbb{Z}=\langle p_{1/n}(t) \rangle \ar[r] & \pi_1(T_{p, q}(1/n))\ar[r]^-{f_n} & \mathbb{Z}_p \ast \mathbb{Z}_q \slash \langle \!\langle (c_1c_2)^{pqn-1} \rangle \! \rangle \ar[r]  & 1, \\
}
\]
where $f \colon G(T_{p, q}) \to G(T_{p, q})/ \langle t \rangle$ and 
$f_n \colon  \pi_1(T_{p, q}(1/n)) \to \pi_1(T_{p, q}(1/n) / \langle t \rangle$ denote the natural projections. 

It follows that for $g\in G(T_{p,q})$, if $f(g)=1$, then $g=t^{m}$ for some $m\neq 0$. Consequently, $p_{1/n}(g) = p_{1/n}(t)^m \neq 1$ for all $n>1$.

Thus we may assume that $f(g) \neq 1$.
By the commutative diagram above, to show $p_{1/n}(g) \neq 1$, it is sufficient to show that $P_{1/n}(f(g)) \neq 1$, which is done in the following lemma.

\begin{lemma}
For $w \in \mathbb{Z}_{p} \ast \mathbb{Z}_{q}$, let $\ell=\ell(w)$ be the length of the normal form $w$. 
Namely, $w = x_1x_2\cdots x_{\ell}$, 
where each $x_i\ne 1$, each $x_i$ is in one of the factors, $\mathbb{Z}_p$ or $\mathbb{Z}_q$, and
successive $x_i$ and $x_{i+1}$ are not in the same factor.
If we take $n$ so that $npq-1 > \ell(w)$, then $P_{1/n}(w) \neq 1$.
\end{lemma}

\begin{proof}
The assertion follows from Small Cancellation Theory. 
Since $npq-1>6$, 
the relator $(c_1c_2)^{npq-1}$ of the free product $\mathbb{Z}_p \ast \mathbb{Z}_q$ satisfies the $C'(\frac{1}{6})$ condition of the small cancellation over the free product. 
Therefore by Greendlinger's lemma for small cancellation over the free product \cite[V. Theorem 9.3]{Lyndon-Schupp},  
if $P_{1/n}(w)=1$, 
then the normal form of $w$ must contain one of 
$(c_1c_2)^{L}$, $(c_2c_1)^{L}, (c_1^{-1}c_2^{-1})^{L}, (c_2^{-1}c_1^{-1})^{L}$ for $L > \frac{1}{2}(npq-1)$. 
When $npq-1 > \ell$, this is impossible, hence $P_{1/n}(w) \neq 1$.
\end{proof}

Therefore by lemma, by taking $N_g=\max\{2, \frac{\ell(f(g))+2}{pq} \}$, $p_{1/n}(g) \neq 1$ whenever $n \geq N_g$.
\end{proof}

\medskip

\section{Fundamental group of graph of groups}
\label{sec:graph of groups}

In this section, 
we will review the fundamental group of graph of groups \cite{Serre}.
For our purpose we will restrict our attention to a special case which appears as the knot group $G(K) = \pi_1(E(K))$ or $\pi_1(K(r))$. 

\subsection{Before Dehn fillings}
\label{graph_before}

Let us put $G_{v} = \pi_1(X) \subset G(K)$ and $G_{v_j} = \pi_1(E_j) \subset G(K)$, each of which is called a {\em vertex group}. 
Similarly put $G_{e_i} = \pi_1(T_i) \subset G(K)$, which we call an {\em edge group}. 

Let $Y_{E(K)}$ be a graph consisting of $n+1$ vertices $v, v_1, \dots, v_n$,  
together with oriented edges $e_1, \dots, e_n$; 
$e_i$ connects $v$ and $v_i$ directed from $v$ to $v_i$. 
By definition, $Y_{E(K)}$ is just a tree as depicted by Figure~\ref{graph_groups}. 

Denote by $o(e_i)$ the vertex $v$ which is called the {\em origin} of the oriented edge $e_i$, and 
denote by $t(e_i)$ the vertex $v_i$ which is called the {\em terminus} of $e_i$. 
Denoting $\overline{e_i}$ the edge $e_i$ with opposite orientation,  
$o(\overline{e_i}) = t(e_i)$ and $t(\overline{e_i}) = o(e_i)$. 

To the distinguished vertex $v$, assign a vertex group $G_{v} = \pi_1(X)$, 
and to each $v_i$ assign a vertex group $G_{v_i} = \pi_1(E_i)$ for $1 \le i \le n$. 
The middle of $e_i$ corresponds to an edge group $G_{e_i}= \pi_1(T_i)$. 
Note that $G_{e_i} \subset G_{v_i}$ and $G_{\overline{e_i}} \subset G_{v}$. 
Then the tree $Y_{E(K)}$, together with a family of groups $G = \{ G_v,\ G_{v_i},\ G_{e_i}\}$, 
gives a \textit{graph of groups} $(G, Y_{E(K)})$.

\begin{figure}[htb]
\centering
\includegraphics[width=0.25\textwidth]{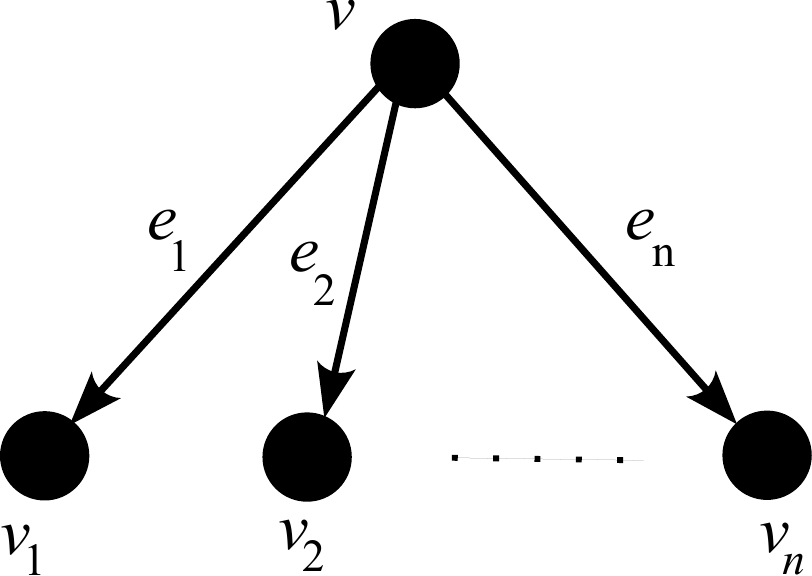}
\caption{Graph of groups.} 
\label{graph_groups}
\end{figure}

Now let us introduce a graphical word of the graph of groups $(G, Y_{E(K)})$.  
Let $c$ be an oriented closed path in $Y_{E(K)}$ based at the vertex $v$, 
which is expressed by a sequence of the oriented edges $(y_1, \dots, y_{\ell})$; 
each $y_i$ is one of $e_1, \overline{e_1}, \dots, e_n, \overline{e_n}$ such that 
$o(y_{i+1}) = t(y_i)$ for each $i$. 
Then a \textit{graphical word} is a sequence 
\[
x_0y_1 x_1y_2 \cdots y_i x_i y_{i+1} \cdots y_{\ell}  x_{\ell},
\]
$x_0, x_1, \dots, x_{\ell}$ of $G_v$ or $G_{v_i} (1 \le i \le n)$ such that 
$x_0, x_{\ell} \in G_v$, and $x_i \in G_{t(y_i)} = G_{o(y_{i+1})}$. 
Note that $x_i$ may be the trivial element $1$.  

Since $Y_{E(K)}$ is a tree, 
for the graph (tree) of groups $(G, Y_{E(K)})$,  
putting $y_j =1$ in the above graphical word, we obtain an element $x_0\cdots x_{\ell}$ in \textit{the fundamental group of $(G, Y_{E(K)})$}, which is $\pi_1(E(K))$, 
the amalgams of the vertex groups along the edge groups; 
see \cite[pp.42--43]{Serre}.

For a non-trivial element $g \in G(K)$ with minimal expression $g = g_1\cdots g_m$, 
we may assign a graphical word in the following manner. 

We start from the word $1 g_1 g_2 \cdots g_m$ if $g_1 \not \in \pi_1(X)=G_{v}$ and $g_1 g_2\cdots g_m$ if $g_1 \in \pi_1(X)=G_v$. 
Here the first symbol $1$ represents the trivial element of $G_v = \pi_1(X)$.
For each $i=1,\ldots,m$, we replace each $g_i$ with sequence according to the following rule.

If $g_i \in \pi_1(X)=G_{v}$ then we remain $g_i$ unchanged. 
We emphasize that $g_i \in \pi_1(E_j)-\pi_1(X) = \pi_1(E_j) - \pi_1(T_j)$ in the cases where modifications are needed. 
Assume that $g_i \in \pi_1(E_{j(i)})= G_{v_{j(i)}}$ for some terminus vertex $v_{j(i)}$.
If $g_{i+1} \in \pi_1(X)=G_v$, then we replace $g_i$ with $e_{j(i)}g_i \overline{e_{j(i)}}$.
If $g_{i+1} \in \pi_1(E_{j(i+1)})$, or, $i=m$, then then we replace $g_i$ with $e_{j(i)}g_i \overline{e_{j(i)}}1$, where the last $1$ denotes the trivial element of $G_v$.

The result is the alternating sequence 
\[ 
x_0y_1 x_1y_2 \cdots y_i x_i y_{i+1} \cdots y_{m}  x_{m+1} 
\]
consisting of elements $x_0, x_1, \dots, x_{m+1}$ of $G_v$ or $G_{v_i} (1 \le i \le n)$ and the sequence of oriented edges $(y_1, \dots, y_m)$ which forms a closed oriented path $c$ based at $v$. 

We call $x_0y_1 x_1y_2 \cdots y_i x_i y_{i+1} \cdots y_{m}  x_{m+1}$ the {\em graphical word associated to $g$} and denote it by $|c, g|$. 

The example below illustrates how we obtain a graphical word from an element with minimal expression. 

\begin{example}
\label{word_sequence}
Assume that $G(K) = \pi_1(E(K))$ consists of four vertices $v, v_1, v_2, v_3$ and three edges; see Figure~\ref{tree}. 
Let $g = g_1g_2g_3g_4$, whee $g_1 \in G_{v_2}= \pi_1(E_2), g_2 \in G_{v_1}=\pi_1(E_1), g_3 \in G_v = \pi_1(X)$ and $g_4 \in G_{v_2}= \pi_1(E_2)$. 
Then the graphical word of $g$ is given by  
\[
x_0 y_1 x_1 y_2  x_2 y_3 x_3 y_4 x_4 y_5 x_5 y_6 x_6  = 1 e_2 g_1 \overline{e_2} 1 e_1 g_2 \overline{e_1} g_3 e_2 g_4 \overline{e_2} 1.
\] 

\begin{figure}[htb]
\centering
\includegraphics[width=0.85\textwidth]{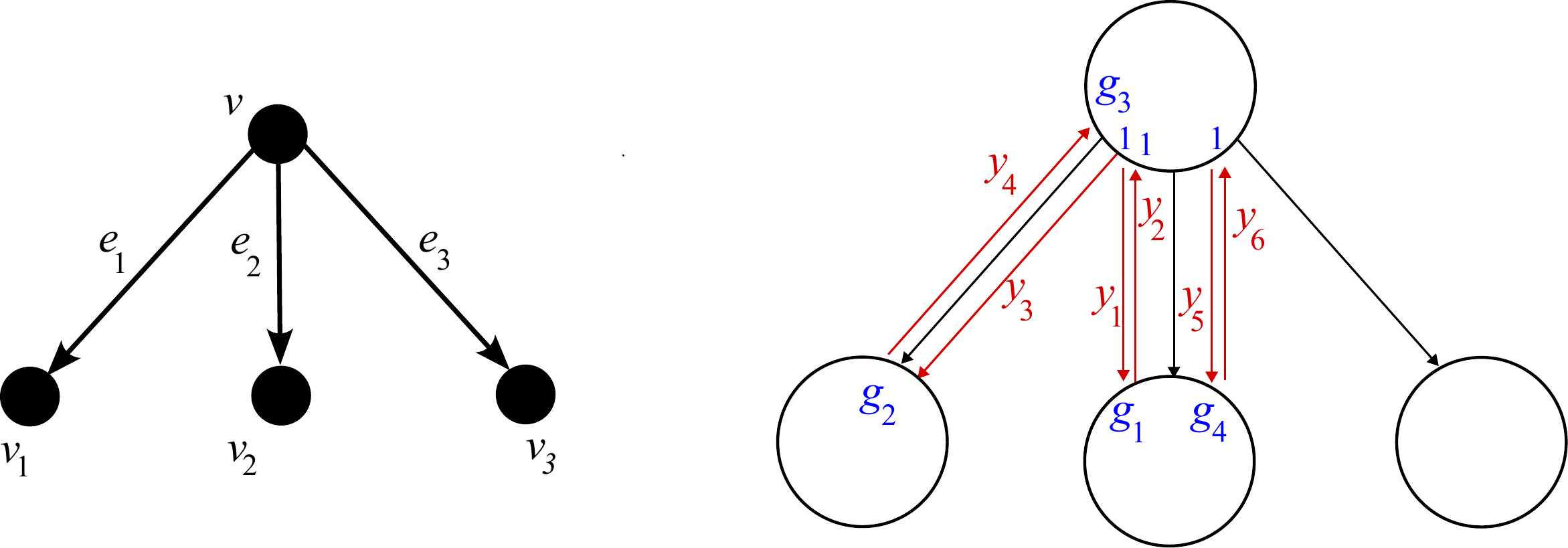}
\caption{A graph of groups $(G, Y_{E(K)})$ and an oriented closed path in $Y_{E(K)}$.} 
\label{tree}
\end{figure}

\end{example}

\subsection{After Dehn fillings}
\label{graph_after}

In what follows we assume $X(r)$ is boundary-irreducible. 
Then for each component $T_j \subset \partial X(r)$, we may regard that $\pi_1(T_j)$ is a subgroup of $\pi_1(X(r)) \subset \pi_1(K(r))$. 

For simplicity we write $G_{\widehat{v}} = \pi_1(X(r))$ and $G_{v_j} = \pi_1(E_j)$ and $G_{e_i} = \pi_1(T_i)$ as in \ref{graph_before}. 
Let $\widehat{G}$ be a family of groups $\{ G_{\widehat{v}},\ G_{v_j},\ G_{e_i}  \}$.
Then we obtain a graph of groups $(\widehat{G}, Y_{K(r)})$ consisting of $n+1$ vertices $\widehat{v}, v_1, \dots, v_n$,  
together with oriented edges $y_1, \dots, y_n$; 
$y_i$ connects $v$ and $v_i$ directed from $v$ to $v_i$; 
$\overline{y_i}$ denotes $y_i$ with the opposite direction.
Note that the graph $Y_{K(r)}$ is exactly the same as $Y_{E(K)}$ except for the vertex group $G_{\widehat{v}} = \pi_1(X(r))$. 

Recall that $r$--Dehn filling of $E(K)$ induces a natural epimorphism 
$p_r \colon G(K) \to \pi_1(K(r))$. 
Let $g$ be a non-trivial element of $G(K)$ with a minimal expression $g_1 \cdots g_m$ with the associated graphical word 
\[
x_0 y_1 x_1 \cdots y_{i-1} x_{i-1} y_i x_{i} y_{i+1} \cdots x_{\ell}. 
\] 

Then 
$p_r(g)$ is $p_r(g_1) \cdots p_r(g_n) \in \pi_1(K(r))$ with the associated graphical word 
\[
p_r(x_0) y_1 p_r(x_1) \cdots y_{i-1} p_r(x_{i-1}) y_i p_r(x_{i}) y_{i+1} \cdots p_r(x_{\ell}).
\] 

Note that $p_r(g_i)$ remains non-trivial if $g_i \in G_{v_j} = \pi_1(E_j)$, 
while if $g_i \in G_v = \pi_1(X)$, 
then $p_r(g_i) = p_{X, r}(g_i) \in  G_{\hat{v}} = \pi_1(X(r))$ may be trivial, 
and hence 
$p_r(g)$ may be also trivial in $\pi_1(K(r))$ for some $r \in \mathbb{Q}$.

We recall the following definition from \cite{Serre}. 
For our purpose we restrict our attention to the graph $Y$, 
which will be either $Y_{E(K)}$ or $Y_{K(r)}$ described in \ref{graph_before} and \ref{graph_after}; 
$Y = Y_{E(K)}$ has vertex groups $G_v = \pi_1(X)$, $G_{v_j} = \pi_1(E_j)$ and edge groups $G_{y_i} = \pi_1(T_j)$, 
and $Y = Y_{K(r)}$ has vertex groups $G_{\widehat{v}} = \pi_1(X(r))$, $G_{v_j} = \pi_1(E_j)$ and edge groups $G_{y_i} = \pi_1(T_j)$. 
In the following definition, each $x_k$ denotes $g_{k'} \in G(K)$ or $p_r(g_{k'}) \in \pi_1(K(r))$, and $y_k$ is an oriented edge in $Y$. 

\begin{definition}[reduced graphical word]
\label{reduced}
Let us consider a graphical word 
\[
x_0 y_1 x_1 \cdots y_{i-1} x_{i-1} y_i x_{i} y_{i+1} \cdots x_{\ell}. 
\]
Let $G_{y_i}^{t(y_i)}$ denote the image of the edge group $G_{y_i}$ in the vertex group $G_{t(y_i)}$. 
Then we say that $x_0 y_1 x_1 \cdots y_{i-1} x_{i-1} y_i x_{i} y_{i+1} \cdots x_{\ell}$ is {\em reduced} if 
it satisfies the following: 
\begin{enumerate}
\renewcommand{\labelenumi}{(\roman{enumi})}
\item If ${\ell} = 0$, then $x_0 \ne 1$.
\item
If ${\ell} \ge 1$, then $x_i \not\in G_{y_i}^{t(y_i)}$ for each $i$ such that $y_{i+1} = \overline{y_i}$. 
\end{enumerate}

\begin{figure}[htb]
\centering
\includegraphics[width=0.24\textwidth]{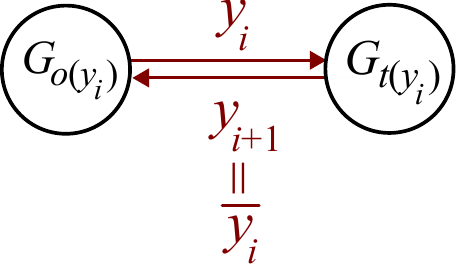}
\caption{backtrack condition: $x_i \not\in G_{y_i}^{t(y_i)} \subset G_{t(y_i)}$ for a backtrack $y_i x_i y_{i+1} = y_i x_i \overline{y_i}$.} 
\label{backtrack}
\end{figure}
\end{definition}

If $y_{i+1} = \overline{y_i}$ and $x_i \in G_{y_i}^{t(y_i)}$, 
then we say that the graphical word $|c, g|$ has an \textit{incomplete backtrack} at $x_i$. 

If $g = g_1 \cdots g_m \in G(K)$ is a minimal expression, then the corresponding graphical word $|c, g|$ is reduced. 

In the proofs of our main results, 
we will use the following theorem due to 
Serre \cite[I.5. Corollary~3]{Serre}. 

\begin{theorem}[\cite{Serre}]
\label{Serre}
Suppose that a graphical word $x_0 y_1 x_1 \cdots y_{i-1} x_{i-1} y_i x_{i} y_{i+1} \cdots x_{\ell}$ is reduced, 
then $x_0  x_1 \cdots  x_{i-1} x_{i} \cdots x_{\ell}$ is non-trivial. 
\end{theorem}

\medskip

\section{Condition for a non-trivial element surviving after Dehn fillings}
\label{proofs}

In what follows we assume that $K$ is a satellite knot, i.e. the outermost decomposing piece $X$ is not $E(K)$. 
In the case where $K$ is non-prime, 
as in Subsection~\ref{outer_piece_composing}, 
we decompose $E(K)$ into $X \cup E_1 \cup E_2$, where $X$ is the $2$--fold composing space, 
instead of the torus decomposition of $E(K)$. 
Let $g$ be a non-trivial element of $G(K)$ and write $g \in G(K)$ in a minimal expression $g_1 \dots g_m$. 
Among $g_1, \dots, g_m$, 
assume that $g_{i_1}, \dots, g_{i_k}$ belong to $\pi_1(X)$; 
possibly some of them belong to $\pi_1(T_j)$ for some $1 \le j \le n$. 

\begin{definition}
\label{g|X}
For a given non-trivial element $g \in G(K)$ with a minimal expression $g= g_1\cdots g_m$ 
for which $g_{i_s} \in \pi_1(X)$ $(s = 1, \dots, k)$,  
we say that a slope $r$ satisfies \textit{Condition $(g|X)$} if the following two conditions hold. 

\begin{enumerate}
\renewcommand{\labelenumi}{(\roman{enumi})}
\item
$X(r)$ is boundary-irreducible.

\item
$p_{X, r}(g_{i_s}) \in \pi_1(X(r))$ is non-trivial, 
and if $g_{i_s} \in \pi_1(X) - \cup_{j=1}^n \pi_1(T_j)$, 
then $p_{X, r}(g_{i_s}) \in \pi_1(X(r))  - \cup_{j=1}^n \pi_1(T_j)$ for $s = 1, \dots, k$. 
\end{enumerate}
\end{definition}

\begin{lemma}
\label{criterion}
Let $g$ be a non-trivial element with a minimal expression $g= g_1\cdots g_m$, in which $g_{i_1}, \dots, g_{i_k}$ belong to $\pi_1(X)$. 
Assume that a slope $r \in \mathbb{Q}$ satisfies Condition  $(g|X)$. 
Then $p_r(g)$ is non-trivial in $\pi_1(K(r))$. 
\end{lemma}

\begin{proof}
Write the associated graphical word  to $g$ as 
\[
|c, g| = x_0 y_1 x_1 \cdots x_{\ell},
\] 
where $y_i$ is an oriented edge, 
and if $x_i \ne 1$, then it is one of $g_1, \dots, g_m$. 
Note that since $g_1\cdots g_m$ is minimal, 
$g_i \ne 1$, and $g_i$ and $g_{i+1}$ do not belong to the same factor group $\pi_1(X)$, $\pi_1(E_j)$ ($1 \le j \le n$). 
Recall also that $g_i \in \pi_1(E_j)$ means $g_i \in \pi_1(E_j) - \pi_1(T_j)$. 
See Figure~\ref{graphical_word}. 

\begin{figure}[htb]
\centering
\includegraphics[width=0.6\textwidth]{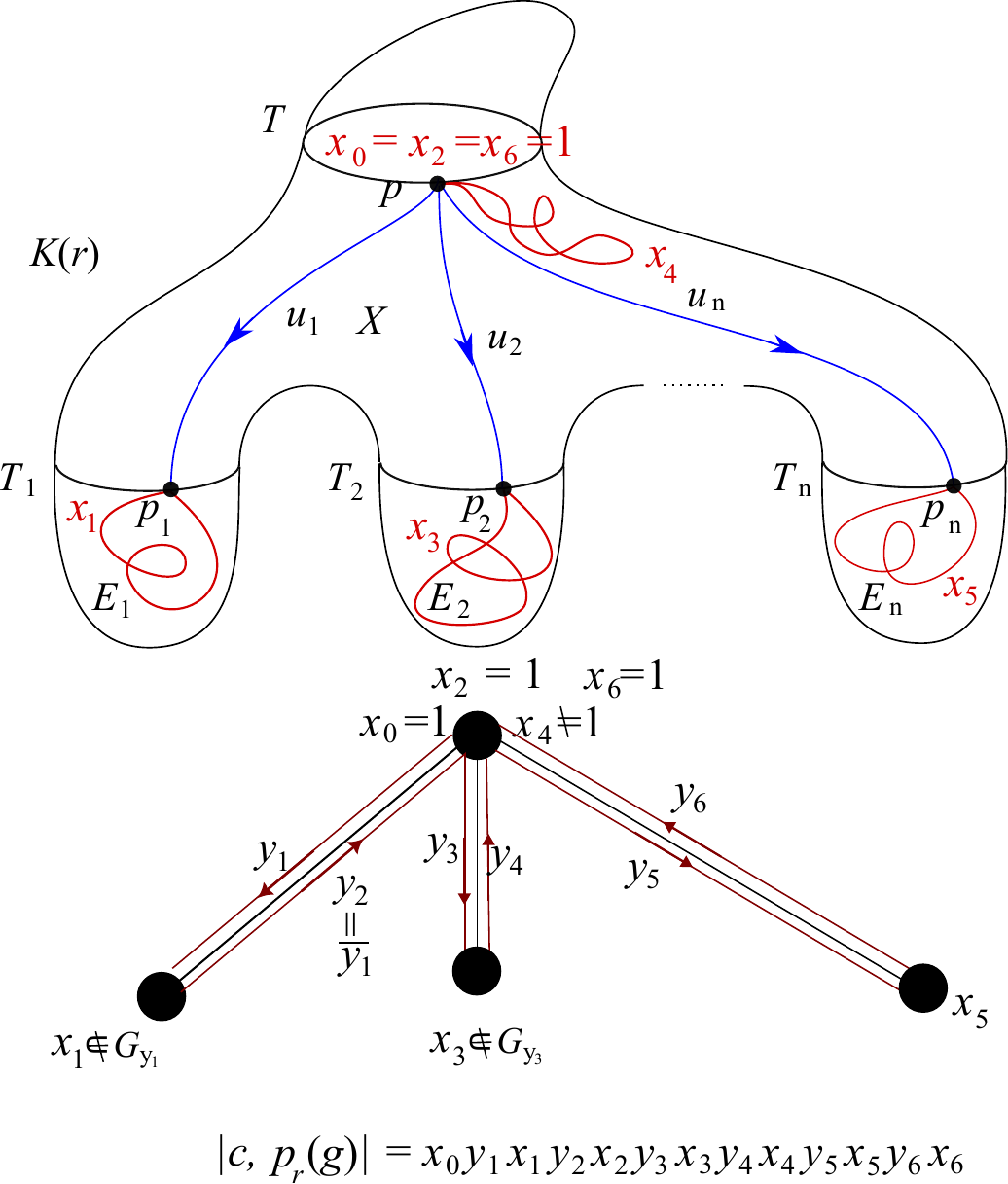}
\caption{$g = g_1g_2g_3g_4$ and its corresponding graphical word 
$|c, g| = x_0 y_1 x_1 y_2 x_2 y_3 x_3 y_4 x_4 y_5 x_5 y_6 x_6$; 
$x_0 = 1, x_1 = g_1, x_2 = 1, x_3 = g_2, x_4 = g_3, x_5 = g_4, x_6 = 1$ and $y_2 = \overline{y_1}, y_4 = \overline{y_3}, y_6 = \overline{y_5}$.}
\label{graphical_word}
\end{figure}

As mentioned before, $|c,g|$ is reduced, because it comes from a minimal expression. 
Then $x_i \in \pi_1(E_j) - \pi_1(T_j)$ if $x_{i-1},\ x_{i+1} \in \pi_1(X)$, 
or $x_i \in \pi_1(X) - \pi_1(T_j)$ if $x_{i-1},\ x_{i+1} \in \pi_1(E_j)$. 

The image of $g$ by the epimorphism $p_r \colon G(K) \to \pi_1(K(r))$ is 
\[
p_r(g) = p_r(g_1) \cdots p_r(g_i) p_r(g_{i+1}) \cdots p_r(g_m).
\] 
The epimorphism $p_r$ sends the graphical word $|c,g|$ to a graphical word
\[
|c, p_r(g)| = p_r(x_0) y_1 p_r(x_1) \cdots y_i p_r(x_i) y_{i+1} p_r(x_{i+1}) \cdots p_r(x_{\ell}).
\]

Under the condition $(g|X)$, an incomplete backtrack in the graphical word $|c, p_r(g)|$ gives rise to an incomplete backtrack in $|c,g|$ and vice versa hence the graphical word $|c, p_r(g)|$ is reduced if and only if $|c,g|$ is reduced. Since we have seen that $|c,g|$ is reduced, this means that $|c,p_r(g)|$ is reduced. 
Hence Theorem~\ref{Serre} shows that $p_r(g)$ is non-trivial in $\pi_1(K(r))$. 
\end{proof}

\medskip

\section{Proofs of Theorems~\ref{prime_non-cable} and \ref{1/n}}
\label{sec:proofs}

\subsection{Proof of Theorem~\ref{prime_non-cable}}
\label{relative_hyperbolic}

For a given non-trivial knot $K$ and $g\in G(K)$, 
define 
\[
\mathcal{S}_K(g) = \{ r \in \mathbb{Q} \mid \textrm{$r$--Dehn filling trivializes $g$}\} \subset \mathbb{Q}.
\]
We will prove that if $K$ is a prime, non-cabled knot, 
then $\mathcal{S}_K(g)$ is finite. 

Suppose that $g$ has a minimal expression $g_1 \cdots g_m$, 
for which $g_{i_1}, \dots, g_{i_k}$ belong to $\pi_1(X)$.  
For each $g_{i_s} \in \pi_1(X)$ $(1 \le s \le k)$,  
following Proposition~\ref{filling_X_hyperbolic}, 
we have a finite subset $\mathcal{S}_X(g_{i_s}) \subset \mathbb{Q}$ such that 
if $r \not\in \mathcal{S}_X(g_{i_s})$, 
then $r$ satisfies conditions (i) and (ii) (for $g_{i_s} \in \pi_1(X)$) in Definition~\ref{g|X}. 
Let us write $\mathcal{B}(g) = \cup_{s=1}^k\mathcal{S}_X(g_{i_s})$, which is also finite. 
Now we show that 
\[
\mathcal{S}_K(g)  \subset \mathcal{B}(g).
\]
Take any $r \in \mathbb{Q}$ not belonging to $\mathcal{B}(g)$. 
Then, by definition,  it enjoys Condition $(g|X)$. 
Hence, by Lemma~\ref{criterion} we see that $p_r(g)$ is non-trivial in $\pi_1(K(r))$. 
This means that $r \not\in \mathcal{S}_K(g)$, and so $\mathcal{S}_K(g)  \subset \mathcal{B}(g)$. 
This completes the implication $(2) \Rightarrow (1)$. 

\medskip

Let us prove the implication $(1) \Rightarrow (2)$. 
To do this we find a non-trivial element in $G(K)$ which can be trivialized by infinitely many Dehn fillings in the case where 
$K$ is a torus knot, a cable of a non-trivial knot or a non-prime knot.

\medskip

\noindent
\textbf{(1) $K$ is a  $(p, q)$--torus knot. }

For $(p, q)$--torus knot $T_{p, q}$, 
the slope $m/n$ with $|pqn - m| = 1$ are cyclic surgery slopes, i.e. $\pi_1(T_{p, q}(m/n))$ is cyclic. 
Then for such slopes $m/n$, 
the epimorphism $G(K) \to \pi_1(K(m/n))$ factors through $G(K) \to G(K)/[G(K), G(K)]$. 
Hence, every element in the commutator subgroup $[G(K), G(K)]$ becomes trivial in $\pi_1(K(m/n))$. 
This means that for each $g \in [G(K), G(K)]$, 
$\mathcal{S}_K(g)$ is infinite. 

\medskip

\noindent
\textbf{(2) $K$ is a  $(p, q)$--cable of a non-trivial knot $(q \ge 2)$. }

Let $X$ be the outermost decomposing piece of $E(K)$, 
which is a cable space, 
a Seifert fiber space over the annulus with an exceptional fiber of index $q$. 
Then for any slope $m/n$ with $|pqn-m| = 1$, 
$X(m/n)$ is a solid torus. 
Thus an epimorphism $\pi_1(X) \to \pi_1(X(m/n))$ factors through $\pi_1(X) \to \pi_1(X)/[\pi_1(X), \pi_1(X)]$. 
This shows that any element $g \in [\pi_1(X), \pi_1(X)]$ becomes trivial in $\pi_1(X(m/n))$, 
in particular, it becomes trivial in $\pi_1(K(m/n))$. 
Hence, for each $g \in [\pi_1(X), \pi_1(X)] \subset [G(K), G(K)]$, 
$\mathcal{S}_K(g)$ is infinite. 

\medskip

\noindent
\textbf{(3) $K$ is a non-prime knot. }

As in Subsection~\ref{outer_piece_composing} we decompose $E(K)$ as $E(K) = X \cup E_1 \cup E_2$ $(i = 1, 2)$, 
where $X$ is a $2$--fold composing space, i.e. $X = [\textrm{disk with two holes}] \times S^1$ and $E_i$ is a non-trivial knot exterior. 
Note that the meridian of $K$ is a regular fiber in $X$. 
For any integral slope $m$, 
$X(m) = S^1 \times S^1 \times [0, 1]$. 
Since $\pi_1(X(m)) = \mathbb{Z} \oplus \mathbb{Z}$, 
an epimorphism $\pi_1(X) \to \pi_1(X(m))$ factors through $\pi_1(X) \to \pi_1(X)/[\pi_1(X), \pi_1(X)]$, 
and any non-trivial element $g \in [\pi_1(X), \pi_1(X)] \subset \pi_1(X) \subset G(K)$ becomes trivial in 
$\pi_1(X(m))$, and hence in $\pi_1(K(m))$. 
Hence, $\mathcal{S}_K(g)$ contains infinitely many slopes for $g \in [\pi_1(X), \pi_1(X)] \subset [G(K), G(K)]$. 

\smallskip

This completes a proof of Theorem~\ref{prime_non-cable}. 
\qed

\medskip

The implication $(2) \Rightarrow(1)$ in Theorem~\ref{prime_non-cable} may be obtained from \cite{GM,Osin}, 
together with \cite[Theorem 7.2.2]{AFW} and \cite{Dahmani}, 
using the hyperbolicity of 3-manifold group relative to the set of parabolic subgroups which come from non-hyperbolic JSJ pieces.

\medskip

\subsection{Proof of Theorem~\ref{1/n}}
\label{non_trivial_knots}

Suppose that $K$ is a prime, non-cabled knot. 
Then the result follows from 
Theorem~\ref{prime_non-cable} ($(2) \Rightarrow (1)$). 

Let us assume that $K$ is a non-trivial knot other than above knots, 
i.e. $K$ is a non-prime knot or a cabled knot. 
In the latter case, we separate into sub-cases depending upon its companion knot is trivial or not. 

Assume first that $K$ is a non-prime knot or a cable of some non-trivial knot. 
Then for each $g_{i_s} \in \pi_1(X)$, 
following Lemmas~\ref{filling_X_composing}, \ref{filling_X_cable}, 
we have a constant $N_{g_{i_s}}$ such that 
if $n \ge N_{g_{i_s}}$, 
then the slope $1/n$ satisfies conditions (i) and (ii) (for $g_{i_s} \in \pi_1(X)$) in Definition~\ref{g|X}. 
Put $N_g = \mathrm{max}\{ N_{g_{i_1}}, \dots, N_{g_{i_k}}\}$. 
Then for any $n \ge N_g$, the slope $1/n$ enjoys Condition $(g|X)$. 
Hence, by Lemma~\ref{criterion} we see that $p_r(g)$ is non-trivial in $\pi_1(K(r))$. 

Finally assume that $K$ is a torus knot. 
Then Lemma~\ref{torus _knot_fully} gives us the desired result. 

Putting $N_F = \mathrm{max}\{ N_g \mid g \in F \}$, 
we finish the  proof of Theorem~\ref{1/n}. 
\qed

\subsection{Remark on hyperbolicity}
\label{hyperbolicity}
Following the proof of Theorem~\ref{1/n}, 
the knot group $G(K)$ is fully residually perfect hyperbolic, toroidal, or Seifert $3$--manifold group according to 
$K$ is a hyperbolic knot, a satellite knot, or a torus knot, respectively. 
If $K$ is a cable of a hyperbolic knot, 
then we have an epimorphism from $G(K)$ to a closed hyperbolic $3$--manifold group. 
However, $G(K)$ cannot be residually closed hyperbolic $3$--manifold group. 
Indeed, we have: 

\begin{claim}
\label{non-residually_hyp}
If $K$ is a torus knot or a satellite knot with non-hyperbolic JSJ-decomposing piece, 
then $G(K)$ cannot be residually closed hyperbolic $3$--manifold group.  
\end{claim}

\begin{proof}
Assume that $K$ is a torus knot. 
Then $E(K)$ is a Seifert fiber space with a regular fiber $t$; 
we use the same symbol $t$ to denote the element of $G(K)$ represented by the oriented regular fiber. 
Let us choose $F = \{ t \}$.
Note that $t$ is central in $G(K)$. 
Suppose for a contradiction we have an epimorphism $\varphi$ from $G(K)$ to $\pi_1(M)$ for some closed hyperbolic $3$--manifold $M$
with $\varphi(t)\ne 1$.  
Then $\varphi(t)$ is a non-trivial, central element in $\pi_1(M)$, a contradiction. 

Now suppose that $K$ is a satellite knot with non-hyperbolic JSJ-decomposing piece $Y$. 
Assume for a contradiction that $G(K)$ is residually closed hyperbolic $3$--manifold group. 
Then following Proposition~\ref{fully_residual=residual}, 
we may assume $G(K)$ is fully residually closed hyperbolic $3$--manifold group. 
Since $Y$ is non-hyperbolic, it is Seifert fibered. 
Let $t$ be a regular fiber in $Y$; as above we use the same symbol $t$ to denote the element of $\pi_1(Y)$ represented by the oriented regular fiber. 
Write $H = \pi_1(Y)$ and take a non-trivial element $g \in [H, H]$. 
Note that $t \not\in [H, H]$. 
Set $F = \{ t, g \} \in \pi_1(Y) \subset G(K)$. 
By our assumption we have a closed hyperbolic $3$--manifold $M$ and an epimorphism
$\varphi \colon G(K) \to \pi_1(M)$ such that $\varphi(t) \ne 1$ nor  $\varphi(g) \ne 1$. 
Let $\rho$ be a holonomy representation $\pi_1(M) \to \mathrm{PSL}_2(\mathbb{C})$. 
Consider $\rho \circ \varphi \colon G(K) \to \mathrm{PSL}_2(\mathbb{C})$. 

Now we show that $\varphi(H)$ is abelian. 
Take any element $h \in H$. 
Assume that $\varphi(h) \ne 1$, hence $\rho \circ \varphi(h) \ne 1$.  
By the assumption $\varphi(t) \ne 1$, and $\rho \circ \varphi(t) \ne 1$. 
Since $t$ is central in $H$, 
$[t, h] = 1$ and hence $[\rho \circ \varphi(t), \rho \circ \varphi(h)] = 1$ in $\mathrm{PSL}_2(\mathbb{C})$. 
Note that $\pi_1(M)$ is torsion free and $(\rho \circ \varphi(t))^2 \ne 1$. 
Then $[\rho \circ \varphi(t), \rho \circ \varphi(h)] = 1$ if and only if 
$\mathrm{Fix}(\rho \circ \varphi(t)) = \mathrm{Fix}(\rho \circ \varphi(h))$; see \cite[4.3]{Beardon}. 
This means that for any non-trivial elements $\varphi(h)$ and $\varphi(h')$ in $\varphi(H)$, 
$\mathrm{Fix}(\rho \circ \varphi(h)) = \mathrm{Fix}(\rho \circ \varphi(h'))$. 
This then implies that $\varphi(h)$ and $\varphi(h')$ in $\varphi(H)$ commute, 
i.e. $\varphi(H)$ is abelian in $\pi_1(M)$. 
Hence, $\varphi|_{H}$ factors $H \to H/[H,H]\to \pi_1(M)$.
However, this contradicts $\varphi(g)\ne 1$.
\end{proof}

\section*{Acknowledgements}
We would like to thank the referee for careful reading and detailed comments, 
and for pointing out errors in the original proofs and suggesting solutions that have allowed us to correct them and improve the paper.

\end{document}